\documentclass[11pt,a4paper]{amsart}
\usepackage[a4paper,centering,left=3.4cm,right=3.4cm]{geometry}
\usepackage{amsfonts,amscd,amssymb,amsmath,amsthm,mathrsfs,multirow,xcolor,lscape}
\usepackage{longtable,pbox,lipsum, stmaryrd,url, epstopdf,float, enumerate}
\usepackage{xspace}
\usepackage[all,cmtip]{xy}
\usepackage[english]{babel}
\usepackage{color}

\usepackage{ucs}
\usepackage[utf8]{inputenc}
\usepackage{dsfont}
\usepackage[T1]{fontenc} 
\usepackage{lmodern}
\usepackage[thinlines]{easytable}

\usepackage{phaistos}
\usepackage{microtype}

\usepackage{enumitem}
\usepackage{scalerel}
\usepackage{todonotes}
\usepackage{stackengine,wasysym}
\usepackage{todonotes}
\usepackage[T1]{fontenc}

\usepackage[colorlinks,linkcolor=blue,anchorcolor=blue,citecolor=blue,backref=page]{hyperref}
\usepackage{hyperref}

\newtheorem{thm}{Theorem}[section]
\newtheorem{thmi}{Theorem}
\newtheorem{defi}[thm]{Definition}
\newtheorem{lemma}[thm]{Lemma}
\newtheorem{prop}[thm]{Proposition}
\newtheorem{coro}[thm]{Corollary}
\newtheorem{rmk}[thm]{Remark}
\newtheorem{conjecture}[thm]{Conjecture}
\newtheorem{conjecture*}{Conjecture}
\newtheorem{notation}[thm]{Notation}

\usepackage[all,cmtip]{xy}
\usepackage{tikz}
\usetikzlibrary{arrows} 
\usetikzlibrary{matrix}
\usetikzlibrary{calc}

\newcommand{\T}{\mathrm{T}}

\newcommand{\rB}{\mathrm{B}}
\newcommand{\rG}{\mathrm{G}}

\newcommand{\rP}{\mathrm{P}}

\newcommand{\rM}{\mathrm{M}}

\newcommand{\rS}{\mathrm{S}}
\newcommand{\rT}{\mathrm{T}}

\newcommand{\mr}[1]{\mathrm{#1}}

\newcommand{\BA}{{\mathbb{A}}}

\newcommand{\BC}{{\mathbb{C}}}

\newcommand{\BG}{{\mathbb{G}}}

\newcommand{\BI}{{\mathbb{I}}}

\newcommand{\BK}{{\mathbb{K}}}

\newcommand{\BQ}{{\mathbb{Q}}}
\newcommand{\BR}{{\mathbb{R}}}
\newcommand{\BS}{{\mathbb{S}}}

\newcommand{\BZ}{{\mathbb{Z}}}

\newcommand{\Fg}{{\mathfrak{g}}}

\newcommand{\CB}{{\mathcal B}}

\newcommand{\CW}{{\mathcal W}}

\newcommand{\one}{\mathds{1}}

\newcommand{\Hom}{\mathop{\rm Hom}\nolimits}
\newcommand{\Ext}{\mathop{\rm Ext}\nolimits}
\newcommand{\Sym}{\mathop{\rm Sym}\nolimits}

\newcommand{\Res}{\mathop{\rm Res}\nolimits}

\newcommand{\Gal}{\mathop{\rm Gal}\nolimits}
\newcommand{\CH}{\mathop{\rm CH}\nolimits}
\newcommand{\CHM}{\mathop{\rm CHM}\nolimits}
\newcommand{\DM}{\mathop{\rm DM}\nolimits}

\newcommand{\MHS}{\mathop{\rm MHS}\nolimits}

\newcommand{\SL}{\mathop{\rm SL} \nolimits}

\newcommand{\GL}{\mathop{\rm GL} \nolimits}

\newcommand{\U}{\mathop{\rm U} \nolimits}
\newcommand{\SU}{\mathop{\rm SU} \nolimits}
\newcommand{\GU}{\mathop{\rm GU} \nolimits}
\newcommand{\Ind}{\mathop{\rm Ind} \nolimits}
\newcommand{\IndUn}{\mathop{\rm IndUn} \nolimits}
\newcommand{\ind}{\mathop{\rm ind} \nolimits}

\newcommand{\der}{\mathop{\rm \tiny{der}} \nolimits}

\newcommand{\Gr}{\mathop{\rm Gr} \nolimits}

\newcommand{\cusp}{\mathop{\rm \tiny{cusp}} \nolimits}

\newcommand{\ord}{\mathop{\rm \tiny{ord}} \nolimits}

\newcommand{\G}{\mathop{{\rm G}} \nolimits}

\newcommand{\K}{\mathop{\tilde{\rm K}_{\infty}} \nolimits}

\newcommand{\Eis}{\mathop{\rm Eis} \nolimits}

\newcommand{\Img}{\mathop{\rm Im} \nolimits}
\newcommand{\dlt}{\mathop{\vert \delta \vert} \nolimits}
\newcommand{\dltv}{\mathop{\vert \delta_v \vert} \nolimits}
\newcommand{\loc}{\mathop{\rm loc} \nolimits}

\newcommand{\id}{\mathop{\rm \tiny{id}} \nolimits}

\newcommand{\Div}{\mathop{\rm Div} \nolimits}

\newcommand{\q}[1]{``#1''}
\newcommand\restr[2]{{
  \left.\kern-\nulldelimiterspace 
  #1 
  \right|_{#2} 
  }}

\title[Bloch-Beilinson for Hecke characters and Picard surfaces]{Bloch-Beilinson conjectures for Hecke characters and Eisenstein cohomology of Picard surfaces}  
\author{Jitendra Bajpai}
\address{Department of Mathematics, Christian-Albrechts University of Kiel, Heinrich-Hecht-Platz 6, 24118 Kiel, Germany}
\email{jitendra@math.uni-kiel.de}
\author{Mattia Cavicchi}
\address{Université Bourgogne Europe, CNRS, IMB UMR 5584, F-21000 Dijon, France}
\email{mattia.cavicchi@ube.fr}
\subjclass[2010]{Primary:11G40;14G35; Secondary:11F75;11F70;14D07}  
\keywords{Eisenstein cohomology, Hecke extension, mixed Hodge structures.}

\begin{document}
\date{\today}

\begin{abstract}
We consider certain families of Hecke characters $\phi$ over a quadratic imaginary field $F$. According to the Bloch-Beilinson conjectures, the order of vanishing of the $L$-function $L(\phi,s)$ at the central point $s=-1$ should equal the dimension of the space of extensions of the Tate motive $\BQ(1)$ by the motive associated with $\phi$.

In this article, we construct candidates for the corresponding extensions of Hodge structures, assuming that the sign of the functional equation of $L(\phi,s)$ is $-1$. This is accomplished through the cohomology of variations of Hodge structures over Picard modular surfaces associated with $F$ and Harder's theory of Eisenstein cohomology.

Furthermore, we demonstrate that these extensions are naturally realized within certain biextensions and outline a program to establish their non-triviality.
\end{abstract}

\maketitle

\tableofcontents


\section{Introduction}\label{intro} 
Let $E$ be an elliptic curve over $\BQ$ and denote by $r$ the rank of the finitely generated abelian group $E(\BQ)$. Suppose that the sign $\epsilon(E)$ of the functional equation of the (completed) $L$-function of $E$ is $-1$. One of the striking consequences of the celebrated Gross-Zagier formula~(\cite{GZ86}), in this case, is that if $L^\prime(E,1) \neq 0$, then $r \geq 1$. Since by the hypothesis on the sign, $L(E,s)$ vanishes at the central point $s=1$, this agrees with the Birch and Swinnerton-Dyer conjecture, which predicts more precisely that $r$ should be \emph{equal} to the order of vanishing of $L(E,s)$ at $s=1$.  

In this article, we are interested in certain conjectural analogues of the latter results, when $L(E,s)$ is replaced by the $L$-function $L(\phi,s)$ of an algebraic Hecke character $\phi$ of a quadratic imaginary field $F$, of odd weight $w$. A natural question arises: what kind of arithmetic information is encoded in the order of vanishing of $L(\phi,s)$ at the central point $s=\frac{w+1}{2}$? The integer $w$ is usually referred to as the motivic weight of the $L$-function.

One way to arrive at a reasonable prediction is by rephrasing the conjecture about elliptic curves in a form that lends itself to generalization. Given a rational point $P$ on an elliptic curve $E$ over $\BQ$, one can consider the open subvariety $U:=E \setminus \{ O, P \}$, where $O$ is the zero point. One looks at the localization exact sequence relating cohomology $H^{\bullet}_Z(E)$ of $E$ supported on $Z:=\{P, O \}$, cohomology with compact supports $H^{\bullet}_c(U)$ of $U$, and cohomology $H^{\bullet}(E)$. Since the divisor $P-O$ gives rise to a class in $H^2_Z(E)(1)$ that maps to 0 in $H^2(E)(1)$, one gets a morphism $\one \rightarrow H^2_Z(E)(1)$ landing in the image of the connecting homomorphism from $H^1_c(U)(1)$. By pullback, this yields an exact sequence
\[
0 \rightarrow H^1(E)(1) \rightarrow E_P \rightarrow \one \rightarrow 0\,.
\]
It turns out that the extension $E_P$ is non-trivial if and only if $P$ is of infinite order (for more details, see the discussion in Remark~\ref{AJ}). This holds both in the sense of extensions of Galois representations (when taking étale cohomology) and of Hodge structures (when taking singular cohomology). Thus, one can reformulate the results of Gross-Zagier as saying that if $\epsilon(E)=-1$ and $L^\prime(E,1) \neq 0$, then there exists a non-trivial extension of $\one$ by $H^1(E)(1)$ \emph{of geometric origin}, i.e. defined through a long exact sequence arising from the cohomology of an algebraic variety. 

As explained in Subsection~\ref{BBHecke_expl}, the Bloch-Beilinson conjectures imply a generalization of the previous statement to the case of an algebraic Hecke character $\phi$ of $F$. To state it, let $\alpha$ denote both the non-trivial element of the Galois group of $F$ and the morphism that it induces on the idèles $\BI_F$ of $F$. Denote by $\vert \cdot \vert_{\BI_F}$ the idelic norm on $\BI_F$, and for a Hecke character $\phi$ of $F$, define the Hecke character $\phi^{\perp}$ by putting
\[
\phi^{\perp}(x):=(\phi(\alpha(x)))^{-1}, \ \ \ x \in \BI_F\,.
\]
\begin{conjecture*}[Conjecture~\ref{BBHecke_Hodge}] \label{Conjecture}
Let $\phi$ be an algebraic Hecke character of $F$ of odd weight w, such that $\phi^{\perp}=\phi \vert \cdot  \vert_{\BI_F}^{w}$. If the sign $\epsilon(\phi)$ of the functional equation of $\phi$ is $-1$, then there exists a non-trivial extension
\[
E_{\phi} \in \Ext^1_{\MHS}\left(\one, H_{\phi}\left(\frac{w+1}{2}\right)\right)
\]
of geometric origin.
\end{conjecture*}
Here, $\MHS$ is a category of \emph{mixed Hodge structures with coefficients}, to be defined in Subsection~\ref{subsec_hodge}, and $H_{\phi}$ is the rank 1 pure Hodge structure associated with $\phi$. It is a sub-object of the Hodge structure on the cohomology of a suitable abelian variety over $F$. The property of being \emph{of geometric origin} will be made precise in Subsection~\ref{motives}; one can think of it as meaning \q{arising from the cohomology of an algebraic variety}. Note that under the hypothesis on $\phi^{\perp}$, there is indeed a well-defined \emph{sign} $\epsilon(\phi)$, and the assumption $\epsilon(\phi)=-1$ implies that the $L$-function of $\phi$ vanishes at the \emph{central point} $s=\frac{w+1}{2}$ (see Remark~\ref{selfadjoint}). 

The purpose of this paper is to initiate a program whose goal is to confirm the above conjecture for certain Hecke characters of weight $-3$. We prove the following result.
\begin{thmi}[Theorem~\ref{sourcext1}.(\ref{ext_mainthm})] \label{mainthm}
Let $k$ be a positive integer, and let $\phi$ be a Hecke character of $F$ of type $(k,-k-3)$ or $(-k-3,k)$, satisfying\footnote{Note the hypothesis on $\phi^{\perp}$ is equivalent to asking that the restriction $\phi_{\BQ}$ of $\phi$ to $\BI_{\BQ}$ verifies
\[
\phi_{\BQ}= \omega_{F \vert \BQ} \vert \cdot \vert^3_{\BI_{\BQ}}
\]
where $\omega_{F \vert \BQ}$ is the order two Dirichlet character associated to the extension $F \vert \BQ$ (see Remark~\ref{meaning_restrcond} for this equivalence). This is the formulation that appears in the statement of Theorem~\ref{sourcext1}. It is the one more adapted to our approach using Eisenstein cohomology.} 
\begin{equation}\label{hyp-phiperp}
\phi^{\perp}=\phi \vert \cdot  \vert_{\BI_F}^{-3} \,.
\end{equation}
If $\epsilon(\phi)=-1$, then there exists an extension
\[
E_{\phi} \in \Ext^1_{\MHS}(\one, H_{\phi}(-1))
\]
of geometric origin.
\end{thmi}

Work in progress of the two authors is aimed at showing that, at least when $L^\prime(\phi,-1) \neq 0$ (a \q{Gross-Zagier type} hypothesis), the extension provided by the above theorem is indeed nontrivial. In the
rest of this introduction, we will first explain the principle of construction of our extension and the evidence for its non-triviality, then give an outline of the next steps of our program. Finally, we will survey the relation with other related works in the literature.

\subsection{Ideas of construction and evidence for non-triviality}
The strategy for proving Theorem 1 is inspired by ideas of Harder~(\cite{Har93}) and makes use of his results in~\cite{Har87b}. Starting from the Hecke character $\phi$, one can construct a subspace $\partial H^2(\phi)$ in degree 2 \emph{boundary cohomology}, with coefficients in an appropriate local system $V_k$, of a \emph{Picard modular surface} $\rS_K$, that is, a non-projective Shimura variety, attached to a $\BQ$-algebraic group $\rG$ such that $\rG(\BR) \simeq \SU(2,1)$. Note that the boundary cohomology of such a variety is computed in terms of the boundary of a compactification, but is independent of the choice of the latter. Then, one would like to lift elements of $\partial H^2(\phi)$ to classes in the cohomology of $\rS_K$ itself. Through Langlands' theory of Eisenstein operators, one sees that the desired lifts can be provided by classes represented by Eisenstein series, well defined as soon as their constant term converges. For Hecke characters satisfying the hypotheses of our Theorem, this convergence is ensured precisely when  $L(\phi,-1)=0$. Hence, in this case, we find that $\partial H^2(\phi)$ is a subspace of \emph{Eisenstein cohomology}, i.e. of the image of the restriction map 
\[
r: H^{\bullet}(\rS_K,V_k) \rightarrow \partial H^{\bullet} (\rS_K,V_k)
\]
from $V_k$-valued cohomology of $\rS_K$ to its $V_k$-valued boundary cohomology.

If $\phi$ is of weight $-3$ and satisfies~\eqref{hyp-phiperp}, then a natural way to enforce the vanishing $L(\phi,-1)=0$ is to ask that $\epsilon(\phi)=-1$. Now, Rogawski's study of the automorphic spectrum of unitary groups in three variables~(\cite{Rog90, Rog92a, Rog92b}) shows that precisely when $\phi$ satisfies this sign hypothesis and~\eqref{hyp-phiperp}, a suitable subspace $H^2(\phi)_!$ can be constructed inside the \emph{interior cohomology} of Picard surfaces with coefficients in the local system $V_k$. Here, interior cohomology is defined as the kernel of the restriction map $r$, and the space $H^2(\phi)_!$ is related to a \emph{cuspidal} $\rG(\BA_f)$-representation induced from $\phi$. 

To put the above two pieces together, one exploits the short exact sequence linking interior cohomology, cohomology, and Eisenstein cohomology, and gets an extension 
\begin{equation} \label{barext_0}
0 \rightarrow H^2(\phi)_! \rightarrow \bar{E}_0 \rightarrow \partial H^2(\phi) \rightarrow 0
\end{equation}
which is in fact an extension of Hodge structures. Classical results allow us to describe the Hodge structure on the piece on the left, while the one on the piece on the right can be computed based on the analysis of the degeneration of variations of the Hodge structure to the boundary of \emph{Baily-Borel compactification} $\rS^*_K$ of $\rS_K$ carried out in~\cite{Anc17}. It turns out that up to a suitable twist, the above extension is nothing but an extension
\[
0 \rightarrow H_{\phi}(-1) \rightarrow E_0 \rightarrow \one^{\oplus \dim \partial H^2(\phi)} \rightarrow 0\,.
\]
Any choice of a class in $\partial H^2(\phi)$ produces then an extension $E_{\phi}$ of geometric origin
\begin{equation} \label{ext_phi}
0 \rightarrow H_{\phi}(-1) \rightarrow E_{\phi} \rightarrow \one \rightarrow 0
\end{equation}
as required to prove our Theorem. 

The above extension is constructed starting from modules under the Hecke algebra with repeating eigenvalues. This happens because of a striking consequence of Rogawski's work: the hypothesis $\epsilon(\phi)=-1$ implies the existence of a sub-Hecke module of interior cohomology of Picard surfaces, isomorphic at all finite places either to an induced representation from an Hecke character, or to a quotient of it. We see this feature of our construction as the main reason to hope that, at least for a careful choice of the class in $\partial H^2(\phi)$, such an extension may be non split. Note that by contrast, the mixed Hodge structure on the first cohomology of a modular curve is split, precisely because, in that case, interior and Eisenstein cohomology have different eigenvalues under the action of the Hecke algebra. 

\subsection{Our program}
With notation as in~\eqref{barext_0}, let $\Psi$ be a class in $\partial H^2(\phi)$, and denote by $E_{\phi}(\Psi)$ the corresponding extension $E_{\phi}$ as in~\eqref{ext_phi}. To confirm Conjecture~\ref{Conjecture}, we need to provide a criterion on $\Psi$ for the non-triviality of $E_{\phi}(\Psi)$, and to check that it is verified for a wise choice of $\Psi$.

As explained in Corollary~\ref{coroext_dual}, there exists a subspace $\partial H^1 (\theta \phi)$ of degree-1 boundary cohomology of the Picard surface $\rS_K$, such that choosing classes $\Psi^\prime$ in $\partial H^1 (\theta \phi)$ yields extensions $E_{\phi}(\Psi^\prime)$ that are \emph{dual} to the $E_{\phi}(\Psi)$'s in a precise sense. We aim to exploit this fact for establishing and verifying the desired criterion, through the following steps: 
\begin{enumerate}
\item the support $\Theta$ of the class $\Psi$, identified with a formal linear combination of cusps in the Baily-Borel compactification of the Picard surface $\rS_K$, provides an algebraic cycle $\mathfrak{E}_{\phi}(\Psi)$ inducing the extension $E_{\phi}(\Psi)$ through an appropriate Abel-Jacobi map; similarly, there exist algebraic cycles $\mathfrak{E}_{\phi}(\Psi^\prime)$ inducing the extensions $E_{\phi}(\Psi^\prime)$;  
\item there is a well-defined \emph{height pairing} which can be evaluated on the algebraic cycles $\mathfrak{E}_{\phi}(\Psi)$, $\mathfrak{E}_{\phi}(\Psi^\prime)$, whose non-vanishing implies that the extension $E_{\phi}(\Psi)$ is non-trivial; 
\item the above height pairing is a (finite) sum of local contributions for each place of $F$, and the contribution from the place at infinity is expressed as the \emph{biextension height} (\cite{Hai90}, \cite{BdJS23}) of a biextension $B_{\phi}(\Psi,\Psi^\prime)$ of Hodge structures having $E_{\phi}(\Psi)$ and $E_{\phi}(\Psi)^\prime$ as subquotients; 
\item the computation of the biextension height of $B_{\phi}(\Psi,\Psi^\prime)$ can be carried out explicitly and yields information about the non-vanishing of the global height pairing. 
\end{enumerate}

This program explains the presence of Section~\ref{biext} of the present paper, devoted to the construction of the biextensions $B_{\phi}(\Psi,\Psi^\prime)$. To make the computation of their biextension height accessible, one needs to provide explicit representatives of cohomology classes in $\partial H^2(\phi)$, by means of differential forms. The construction of such explicit representatives, using the language of Lie algebra cohomology, has recently been achieved in the work~\cite{BC24} by the two authors. The latter should hence be seen as providing one of the crucial ingredients required to complete the above program. Note that it is precisely in the computation of the biextension height that the term $L^\prime(\phi,-1)$ is expected to appear. Its non-vanishing should then be necessary (and sufficient) for the non-vanishing of the pairing itself, and hence for the non-vanishing of our extension class. This will be the subject of the forthcoming article~\cite{BC}.

\subsection{Related work and further future directions}
When $\phi$ is an algebraic Hecke character of odd weight $w$, the Bloch-Beilinson conjectures (\cite{Blo84}, \cite{Bei87}) are much more precise than their consequence studied in this paper, i.e., than Conjecture~\ref{Conjecture}. Namely, they say that the order of vanishing of $L(\phi,s)$ at the central point $s=\frac{w+1}{2}$ should be equal to the (conjecturally finite) rank of a group of algebraic cycles, and they predict the value of the first non-zero Taylor coefficient $L^*(\phi,\frac{w+1}{2})$ of $L(\phi,s)$ at the central point, up to an algebraic number. 

Note that $s=\frac{w+1}{2}$ is a \emph{critical} point, in the sense of Deligne~\cite{Del79}, and hence there exists a canonical \emph{period} associated to it as in \emph{op. cit.}. Moreover, Deligne predicted that when non-zero, the value of $L(\phi,s)$ at such a point should be an algebraic multiple of his period. Deligne's conjecture was later proven for algebraic Hecke characters of $CM$-fields by Blasius~\cite{Bla86}, and the recent preprint~\cite{Kuf24} proposes a proof for general number fields. 

When the order of vanishing of $L(\phi,s)$ at the central point is greater than zero, the algebraic cycles whose existence is predicted by Bloch and Beilinson should give rise to non-trivial extensions, both as Hodge structures and Galois representations, through suitable Abel-Jacobi maps. Moreover, the \q{transcendental part} of the value $L^*(\phi,\frac{w+1}{2})$ should be given by the product of Deligne's period and of the determinant of a \emph{height pairing}. In this paper, we propose candidates for the expected extensions; moreover, the program that we have outlined in the previous section of this introduction should make it possible to compute their (conjecturally non-zero) extension class through a \emph{biextension height}. Our further goal is, of course, to show that this biextension height is related to the aforementioned height pairing and that it bears the expected relationship with the value $L^*(\phi,\frac{w+1}{2})$.  

Galois representation-theoretic extensions have been constructed in~\cite{BC04}, under the hypothesis of \emph{odd} order of vanishing of $L(\phi,s)$ at the central point, for essentially the same kind of Hecke characters $\phi$ of $F$ that we consider here (the case of \emph{even} order of vanishing is studied in~\cite{Her19}). As here, one crucial input in \emph{loc. cit.} is represented by Rogawski's results, and the source of the Galois representations considered in these works is again the cohomology of Picard modular surfaces. On the other hand, to obtain the desired extensions, the authors use $p$-adic families of automorphic forms, without looking at the geometry of the boundary of the Baily-Borel compactification and at Eisenstein cohomology. Hence, their extensions are not of geometric origin in our sense. It would be interesting to study the Galois-theoretic counterpart of our constructions, via \'etale cohomology, and to put it in relation to their methods. 

\subsection{Structure of the article} In Section~\ref{se:preli} we gather all the preliminary facts that we need about the structure of the group $\rG$, its representation theory, its parabolic subgroups, and about its associated locally symmetric spaces, i.e., the Picard surfaces $\rS_K$, together with the variations of Hodge structures carried by them. 

In Section~\ref{Heckechar}, we set up our notations for algebraic Hecke characters. Moreover, we recall the basic facts about their $L$-functions and their associated motives that we need to formulate the Bloch-Beilinson conjectures and to explain how Conjecture~\ref{Conjecture} follows from them. 

In Section~\ref{se:nerve}, we recall the structure of the Baily-Borel and Borel-Serre compactifications of Picard modular surfaces. Moreover, we revisit in detail the computation of boundary cohomology of Picard surfaces and its Hodge structure~(\cite{Anc17}), adding the complements we need. We proceed then by recalling in detail the contents of Harder's paper~\cite{Har87b} on which we rely crucially: in Subsection~\ref{ss:autbound} we record the description of boundary cohomology, using $\bar{\rS}_K$, in terms of induced representations from Hecke characters, and in Section~\ref{se:eis} we explain the description of Eisenstein cohomology, where the $L$-functions of the aforementioned Hecke characters play an essential role. We conclude the section by constructing the first piece of our extensions, the space $\partial H^2(\phi)$. 

In Section~\ref{se:intcoh}, we explain how to use Rogawski's work to construct the second piece needed for our extensions, i.e., the space $H^2(\phi)_!$. The sought-for extensions are then constructed in Section~\ref{se:Heckext}. The final Section~\ref{biext} is devoted to showing how our extensions fit into suitable biextensions. These, as explained before, will provide the starting ground for our future work on the proof of Conjecture~\ref{Conjecture}.


\section{Preliminaries}\label{se:preli}
The first part of this section fixes notation about the ground field $F$ and the algebraic group $\rG$ that will be in force throughout the whole paper. In the following parts, we discuss the root system and Weyl group of $\rG$, we fix a parametrization of its irreducible representations, and we explain the structure of its parabolic subgroups. Finally, we introduce the main geometric objects of interest: the locally symmetric spaces attached to $\rG$, their structure of Shimura varieties, and the variations of Hodge structure associated to representations of $\rG$.

\subsection{Basic notations and conventions}\label{basic} 
We fix a quadratic imaginary extension $F$ of $\BQ$. The only non-trivial element of $\Gal(F \vert\BQ)$ will be denoted by $\alpha$. For $x \in F^{\times}$, its norm $x \cdot \alpha(x)$ will be denoted by $\vert x \vert^2$. 

The symbols $\BA_{\BK}$ and $\BI_{\BK}$ will denote respectively adèles and idèles over a number field $\BK$ (no subscript if $\BK = \BQ$). The quadratic character of $\BI_F$ associated by class field theory to the extension $F \vert \BQ$ will be denoted by $\omega_{F \vert \BQ}$.

We fix as well a 3-dimensional $F$-vector space $V$, and a $F$-valued, non degenerate hermitian form $J$ on $V$, such that $J \otimes_{F} \BC$ is of signature $(2,1)$. For an extension $\mathbb{K}$ of $\BQ$, we will denote by $V_{\mathbb{K}}$ the $\mathbb{K}$-vector space $V \otimes_{\BQ} \mathbb{K}$. 

There always exists an $F$-basis of $V$ with respect to which $J$ is represented by a matrix of the form 
\[
\left(
\begin{array}{ccc}
& & 1 \\
& \beta & \\
1 & &
\end{array}
\right)
\]
where $\beta$ is a strictly positive rational number. We call such a basis a \emph{parabolic basis}. On the other hand, we call \emph{standard basis} a basis of $V_{\mathbb{\BR}}$ in which $J \otimes_{F} \BC$ is represented by the matrix 
\[
\left(
\begin{array}{ccc}
1 & &  \\
& 1 & \\
 & & -1
\end{array}
\right)\,.
\]
A change-of-basis matrix between two such bases is given by 
\begin{equation} \label{parab_to_diag}
\gamma:=\left(
\begin{array}{ccc}
\frac{1}{\sqrt2} & & \frac{1}{\sqrt2} \\
& \sqrt{\beta} & \\
\frac{1}{\sqrt2} & & -\frac{1}{\sqrt2}
\end{array}
\right)\,.
\end{equation}

To the datum of $(V,J)$, we attach the algebraic group over $\BQ$
\[
\rG:=\GU(V,J)
\]
characterized by the property that for every $\BQ$-algebra $R$,
\[
\rG(R)= \{ g \in \GL_{F \otimes_{\BQ} R}(V \otimes_{\BQ} R) \vert \ \exists \ \nu(g) \in R^{\times}, J(g \cdot, g \cdot)=\nu(g)J(\cdot, \cdot) \}\,.
\]
The resulting character $\nu: \rG \rightarrow \BG_m$ is called the \emph{multiplier}. 

We will also need to work with the kernel of the multiplier map
\[
\rG^{\nu = 1} := \U(V,J)
\]
and with the derived subgroup of $\rG$,
\[
\rG^{\der} := \SU(V,J)\,.
\]

\subsection{Structure Theory of $\rG$}\label{sl2}
To describe the structure of $\rG$, recall that $V_F$ decomposes as 
\begin{equation}\label{vbchange}
V_F \simeq V^+ \oplus V^-
\end{equation}
where
\begin{align*}
& V^+:=\{ v \in V\otimes_{\BQ}F \ \vert \ (x \otimes 1)v=(1 \otimes x)v \ \forall x \in F \} \,,\\
& V^-:= \{ v \in V\otimes_{\BQ}F \ \vert \ (x \otimes 1)v=(1 \otimes \alpha(x))v \ \forall x \in F \}.
\end{align*}
The action of $\alpha$ on $V_F$, $v \otimes f \mapsto v \otimes \alpha(f)$, exchanges $V^+$ and $V^-$. The restrictions to $V$ of the projections on $V^+$ and on $V^-$ induce a $F$-linear and a $F$-antilinear isomorphism of $\BQ$-vector spaces
\begin{equation} \label{linear}
V \simeq V^+ , \quad V \simeq V^-\,.
\end{equation}

Now fix any $F$-basis of $V$ and denote by the same symbol $J$ the matrix representing $J$. Under the above identifications, by definition of $\rG$, the relation 
\begin{equation} \label{conj_action}
g_{\vert V^-} = \nu(g) J^{-1} {}^tg^{-1}_{\vert V^+} J
\end{equation}
holds for any $g \in \rG(F)$. This yields an isomorphism
\begin{equation} \label{isoGL3}
\rG_F \simeq \GL(V_F) \times \BG_{m,F}
\end{equation}
by sending $g \in \rG(F)$ to $(g_{\vert V^+},\nu(g))$. It restricts to an isomorphism 
\begin{equation} \label{isoSL3}
\rG^{\der}_F \simeq \SL(V_F)\,.
\end{equation}

Now fix a \emph{parabolic basis} of $V$ (see conventions). Consider the resulting embedding
\[
\rG^{} \hookrightarrow \Res_{F \vert \BQ} \GL_{3,F}
\]
and call $\CB$ and $\mathcal{T}$ the standard (upper triangular) Borel and the maximal torus of $\Res_{F \vert \BQ} \GL_{3,F}$. Then,
\[
\rB := \CB \cap \rG
\]
is a Borel of $\rG$, and
\[
\rT := \mathcal{T} \cap \rG
\]
is a maximal torus that we call \emph{standard}. It is isogenous to $\Res_{F \vert \BQ} \BG_{m,F} \times \Res_{F \vert \BQ} \BG_{m,F}$. Under the isomorphism
\[
\rG_F \simeq \GL_{3,F} \times \BG_{m,F}
\]
deduced from~\eqref{isoGL3}, we get an identification
\begin{equation} \label{maxtorus}
\rT_F(F) \simeq \left\{ 
( \left(
\begin{array}{ccc}
a & & \\
& a^{-1}b & \\
& & b^{-1} q
\end{array}
\right), f) \vert a,b,q,f \in F^{\times} \right\}\,.
\end{equation}
A maximal torus of $\rG^{\der}$ is then given by $\rT \cap \rG^{\der}$. We denote it by
\[
\rT^{\prime} := \rT \cap \rG^{\der}\,.
\] 
It is isomorphic to $\Res_{F \vert \BQ} \BG_{m,F}$, and we fix the isomorphism to be the one described on $\BQ$-points by\footnote{We adopt this convention to be coherent with~\cite[p.~573]{Har87b}.}
\begin{equation} \label{maxtorus_der}
\rT^{\prime} (\BQ) \simeq \left\{ 
\left(
\begin{array}{ccc}
\alpha(a) & & \\
& \alpha(a)^{-1} a & \\
& & a^{-1}
\end{array}
\right) \vert a \in F^{\times} \right\}\,.
\end{equation}
Inside $\rT^{\prime}$, a maximal $\BQ$-split torus $\rT^s$ of $\rG^{\der}$, isomorphic to $\BG_m$, is then described as 
\begin{equation} \label{maxtorus_Qsplit}
\rT^s(\BQ) \simeq \left\{ 
\left(
\begin{array}{ccc}
t & & \\
& 1 & \\
& & t^{-1}
\end{array}
\right) \ \vert \ t \in \BQ^{\times}
\right\}\,.
\end{equation}

\subsection{Root System}\label{roots}

We parametrize characters on the standard maximal torus $\rT_F$ of $\rG_{F}$ by vectors $(k_{1},k_{2},c,r)$ of integers such that
\[
k_{1} \geq k_{2} \geq 0, \
c = k_{1}+k_{2}\pmod 2, \
r = \frac{c+k_{1}+k_{2}}{2}\pmod 2\,.
\]
The parametrization is defined by associating $(k_{1},k_{2},c,r)$ with the character 
\begin{equation} \label{paramchar}
\left( \left(
\begin{array}{ccc}
a & & \\
& a^{-1}b & \\
& & b^{-1} q
\end{array}
\right), f\right) \mapsto a^{k_{1}-k_{2}}b^{k_{2}}q^{\frac{c-(k_{1}+k_{2})}{2}} \cdot f^{-\frac{1}{2}(r+\frac{3c-(k_{1}+k_{2})}{2})}\,.
\end{equation}
\begin{rmk} \rm
The reason for this choice of parametrization comes from Hodge theory, see Remark~\ref{hodge_paramchar}.
\end{rmk}
There is a isomorphism between the subgroup of characters $\lambda$ of $\rT_F$ of the form 
\begin{equation}\label{defcharbig}
(k_{1},k_{2},k_{1}+k_{2},(k_{1}+k_{2}))
\end{equation}
and the group of characters of the maximal torus $\rT_F^\prime$ of $\rG_F^{\der}$. The isomorphism is provided by the natural factorization of characters of the above shape through $\rT^\prime_F$. We still call the resulting character $\lambda$ and denote it by $(k_1,k_2)$. 

The root system $\Phi$ of the split semi-simple group $\rG^{\der}_{F}$ is of type $A_2$. Under the previous identifications, the system $\Delta$ of simple roots is the set $\{(1,-1), (1,2) \}$ and the set $\Phi^+$ of positive roots is the set $\{ (1,-1), (1,2), (2,1) \}$. 

The half-sum of positive roots (and its lift to $\rT_F$ defined via~\eqref{defcharbig}) will be denoted by 
\begin{equation} \label{halfsum}
\delta:= \frac{1}{2}\sum_{\alpha \in \Phi^+} \alpha =(2,1)
\end{equation}
in what follows.

\subsection{Irreducible Representations}\label{irreducible}
Irreducible representations of $\rG^{\der}_{F}$ are identified with irreducible representations of $\SL_{3,F}$, and are therefore parametrized by characters $(k_{1}, k_{2})$ which are \emph{dominant}, i.e. such that $k_{1} \geq k_{2} \geq 0$. Irreducible representations of $\rG_F$ are then parametrized by characters $(k_1,k_2,c,r)$ such that $(k_1,k_2)$ is dominant; the integers $c$ and $r$ then account for a suitable twist by a power of the determinant and a power of the multiplier. We denote by $V_{\lambda}$ the irreducible representation of $\rG_F$ of highest weight $\lambda$. 

The $k$-th symmetric power $\Sym^k V$ of the standard representation $V$ of $\rG_F^{\der}$ corresponds to $\lambda=(k,0)$. We denote by $V_k$ its lift\footnote{See Remark~\ref{univ_abvar_rep} for an explanation of this choice of lift.} to the irreducible representation $V_{(k,0,k,k)}$ of $\rG_F$. We denote by $V^{\vee}_k$ its dual. 
\begin{rmk} \label{conj_highestweight} \rm
The hermitian form $J$ provides a $F$-antilinear isomorphism 
\[
V_k \simeq V^{\vee}_k\,.
\] 
From this, and from relation~\eqref{conj_action}, one sees that $V^{\vee}_k$ corresponds to the irreducible representation $V_{(k,k,0,k)}$.
\end{rmk}

We will also need to parametrize representations of $\rT^\prime$. Consider the isomorphism $$\rT^\prime \simeq \Res_{F \vert \BQ} \BG_{m,F}$$ given by~\eqref{maxtorus_der}, and the natural embedding $$\rT^\prime \rightarrow \Res_{F \vert \BQ} \rT^\prime_F.$$ Fix a character $$\chi: \T^{\prime}_F \rightarrow \BG_{m,F}$$ whose restriction under the above embedding is the map given on $\BQ$-points by  
\[
\left(
\begin{array}{ccc}
\alpha(a) & & \\
& \alpha(a)^{-1} a & \\
& & a^{-1}
\end{array}
\right) \mapsto a^{\mu}\alpha(a)^{\nu}\,.
\]
We then denote by
\begin{equation} \label{I_mu,nu}
I_{\mu,\nu}
\end{equation}
the one-dimensional $F$-vector space on which $F^{\times}$ acts by multiplication via $\chi$.

\subsection{Standard $\mathbb{Q}$-Parabolic Subgroups}\label{standard}

The only $\BQ$-conjugacy class of parabolic subgroups of $\rG$ is the class of the Borel subgroup $\rB$. Indeed, on one hand, it is clear that $\rB$ is a $\BQ$-parabolic. On the other hand, any other parabolic has to stabilize a subspace $W$ inside $V$, hence its orthogonal $W^{\perp}$, and by dimension considerations we see that $W \cap W^{\perp}$ has to be a line. Hence, any $\BQ$-parabolic $\rP$ is the stabilizer of an isotropic line, and since there is a transitive action of $\rG(\BQ)$ on the set of isotropic lines, $\rP$ has to be conjugated to $\rB$.

The Levi component of the maximal parabolic $\rB$ is precisely the torus $\rT$. We denote by $\mathrm{U}$ the unipotent radical of $\rB$, so that we have an isomorphism
\[
\rB \simeq \rT \rtimes \mathrm{U}\,.
\]

\subsection{Kostant Representatives}\label{KostantRep}

We will need to use a classical result of Kostant. To state it, fix a split reductive group $\mathcal{G}$ over a field of characteristic zero, with root system $\mathfrak{r}$ and Weyl group $W$. Denote by $\mathfrak{r}^+$ the set of positive roots and fix moreover a parabolic subgroup $\mathcal{P}$ with its unipotent radical $\mathcal{U}$. Let $\mathfrak{u}$ be the Lie algebra of $\mathcal{U}$ and $\mathfrak{r}_{\mathcal{U}}$ the set of roots appearing inside $\mathfrak{u}$ (necessarily positive). For every $w \in W$, we define:
\begin{align}\label{subweyl}
&\mathfrak{r}^+(w):=\{\alpha \in \mathfrak{r}^+ \vert w^{-1} \alpha \notin \mathfrak{r}^+ \},\nonumber\\
& l(w):=\vert \mathfrak{r}^+(w) \vert,\nonumber\\
&W^{\prime}:= \{ w \in W \vert \mathfrak{r}^+(w) \subset \mathfrak{r}_{\mathcal{U}} \}\nonumber.
\end{align}   

Then, Kostant's theorem reads as follows:
\begin{thm}\cite[Thm.~3.2.3]{Vog81}\label{KostThm}
Let $\mathcal{V}_{\lambda}$  be the irreducible $\mathcal{G}$-representation of highest weight $\lambda$, and let $\delta$ be the half-sum of the positive roots of $\mathcal{G}$. Then, as $(\mathcal{P}/\mathcal{U})$-representations,
\[
H^q(\mathcal{U}, \mathcal{V}_{\lambda}) \simeq \bigoplus \limits_{w \in W^{\prime} \vert l(w)=q} \mathcal{V}^{\mathcal{P} / \mathcal{U}}_{w.(\lambda+\delta)-\delta}
\]
where $\mathcal{V}^{\mathcal{P} / \mathcal{U}}_{\mu}$ denotes the irreducible $(\mathcal{P}/\mathcal{U})$-representation of highest weight $\mu$. 
\end{thm} 

Coming back to our situation, the Weyl group $W$ of $\rG_{F}$ is isomorphic to the symmetric group on 3 elements. Fix a character $\lambda=(k_{1},k_{2},c,r)$ of $\rT_F$. Then, for an element $w \in W$, one computes that the \q{twisted} Weyl action 
\begin{equation} \label{KostWeilaction}
w \star \lambda:=w.(\lambda+\delta)-\delta
\end{equation}
appearing in Kostant's theorem is given as follows:
\begin{align*}
id \star{\lambda} & =\lambda \\
(1\ 2) \star \lambda &= (k_{2}-1,k_{1}+1,c,r) \\
(2\ 3)\star \lambda &= (k_{1}-k_{2}-1,-k_{2}-2,c-k_{2}-1,r) \\
(1\ 2\ 3)\star \lambda &= (-k_{2}-3,k_{1}-k_{2},c-k_{2}-1,r) \\
(1\ 3\ 2)\star \lambda &= (k_{2}-k_{1}-3,-k_{1}-3,c-k_{1}-2,r) \\
(1\ 3)\star \lambda &= (-k_{1}-4,k_{2}-k_{1}-2,c-k_{1}-2,r)
\end{align*}
The only element of length 0 is the identity, the elements $(1\ 2)$ and $(2\ 3)$ have length 1, the elements $(1\ 2\ 3)$ and $(1\ 3\ 2)$ have length 2, and the element $(1\ 3)$ has length 3. 

Notice that the restricted Weyl group of $\rG$ is isomorphic to $\mathbb{Z}_2$. We fix a generator of it, whose lift $\theta$ to $\rG$ is given by (in a parabolic basis)
\begin{equation} \label{longestWeyl}
\theta:=\left(
\begin{array}{ccc}
 & & -1 \\
 & -1 & \\
 -1 & &
\end{array}
\right)\,.
\end{equation}

\subsection{The associated symmetric spaces and Shimura varieties}\label{sideremarks} 
To describe the locally symmetric spaces associated with $\rG^{\der}$, observe that $\rG^{\der}(\BR) \simeq \SU(2,1)$. Fix a standard basis of $V_{\BR}$, consider the resulting embedding $\rG^{\der}(\BR) \hookrightarrow \GL_3(\BC)$ and write an element $g \in \rG(\BR)$ in block-matrix form as 
\[
\left(
\begin{array}{cc}
A & B \\
C & D
\end{array}
\right)
\]
where $A$ is a 2-by-2 matrix, $B$ and $C$ are respectively a 2-by-1 and a 1-by-2 matrix, and $D$ is a scalar. There is a standard maximal compact subgroup $K_{\infty}$ of $\rG^{\der}(\BR)$, given by the subgroup of elements
\begin{equation} \label{maxcomp}
\left\{
\left(
\begin{array}{cc}
A & \\
 & \det(A)^{-1}
\end{array}
\right)
\ \vert \ A \in \mathrm{U}_2(\BR)
\right\}
\end{equation}
where $\det A \in \mathrm{U}_1(\BR) \simeq S^1$. Then, $K_{\infty} \simeq \mathrm{U}_2(\BR)$, and each other maximal compact subgroup of $\rG^{\der}(\BR)$ is conjugated to $K_{\infty}$. 

The symmetric space associated to $\rG^{\der}$ is then 
\[
\rS := \rG^{\der}(\BR)/K_{\infty}\,.
\]
The Lie group $\rG^{\der}(\BR)$ acts transitively by generalized fractional transformations on the \q{matrix} space 
\[
\mathcal{D}_{2,1}:= 
\{ U \in M_{2,1}(\BC) \vert \bar{U}U-1 < 0 \}
\]
which is nothing but a complex 2-ball in $\BC^2$. The stabilizer of 0 being precisely $K_{\infty}$, the action defines a diffeomorphism between $\rS$ and $\mathcal{D}_{2,1}$. This yields a complex analytic structure on $\rS$. Then, to any arithmetic subgroup $\Gamma \subset \rG^{\der}(\BQ)$, one associates the locally symmetric space 
\[
\rS_\Gamma:= \Gamma \backslash \rS
\]
and the complex analytic structure on $\rS$ gives $\rS_\Gamma$ the structure of a complex analytic variety.

When $\Gamma$ is a congruence subgroup, the above locally symmetric spaces can be seen as (analytifications of connected components of) Shimura varieties. To see this, one fixes a parabolic basis of $V$, considers the \emph{Deligne torus} $\BS:= \Res_{\BC \vert \BR} \BG_{m,\BR}$ and defines the morphism $h: \BS \rightarrow \rG_{\BR}$ on $\BR$-points by
\begin{equation} \label{Shim_datum}
x+iy \mapsto \left(
\begin{array}{ccc}
x & 0 & iy \\
0 & x+iy & 0 \\
iy & 0 & x
\end{array}
\right) \,.
\end{equation}
It defines a \emph{pure Shimura datum} (~\cite[Sect.~3.2]{Lan17}\footnote{To make the comparison with Lan's setting, one has to notice that $i \cdot \big(\begin{smallmatrix}  & & 1\\ & 1 & \\ i & & \end{smallmatrix}\big) \cdot J_{\beta} \cdot \big(\begin{smallmatrix}  & & -i\\ & 1 & \\ 1 & & \end{smallmatrix}\big)=\big(\begin{smallmatrix}  & & 1\\ & i \beta & \\ -1 & & \end{smallmatrix}\big)$, where the latter matrix represents a \emph{skew-hermitian} form, and that conjugating the above Shimura datum by $\big(\begin{smallmatrix}  & & -i\\ & 1 & \\ 1 & & \end{smallmatrix}\big)$ gives the \emph{complex conjugate} of Lan's Shimura datum (\emph{loc. cit.}, Eq. (3.2.5.4)).}). Then, the $\rG(\BR)$-conjugacy class $X$ of $h$ is in canonical bijection with a disjoint union of copies of $\rS$ and therefore acquires a canonical topology. For any compact open subgroup $K \subset \rG(\BA_f)$ which is \emph{neat} 
in the sense of~\cite[Sect.~0.6]{Pin90}, the set of double classes
\begin{equation*}
\rG(\BQ) \backslash (X \times \rG(\BA_f) /K)
\end{equation*}
equipped with its natural topology, is canonically homeomorphic to a finite disjoint union of locally symmetric spaces
\begin{equation} \label{locsymspaces}
\bigsqcup_i \rS_{\Gamma_i}
\end{equation}
where the $\Gamma_i$'s are \emph{neat} (in particular torsion-free) \emph{congruence} subgroups of $\rG(\BQ)$, given by 
\[
\Gamma_i=g_i K g_i^{-1} \cap \rG(\BQ)
\]
for suitable elements $g_i \in \rG(\BA_f)$. As a consequence, the above set of double classes acquires a canonical structure of complex analytic space, and can be shown to be isomorphic to the analytification of a canonical \emph{smooth} and \emph{quasi-projective} algebraic surface $\rS_K$, called a \emph{Picard surface} of \emph{level} $K$, defined over $F$. It satisfies
\[
\rS_K(\BC) = \rG(\BQ) \backslash (X \times \rG(\BA_f) /K)\,.
\]
\begin{rmk} \label{univ_abvar} \rm
Each surface $\rS_K$ can be described as a moduli space of abelian threefolds equipped with additional structures depending on $K$, amongst which complex multiplication by the ring of integers of $F$. In particular, there exists a universal abelian threefold $A_K$ over $\rS_K$.
\end{rmk}

The locally symmetric spaces $\Gamma_i \backslash \rS$ are identified with the analytifications of \emph{smooth}, geometrically connected quasi-projective surfaces, defined over an abelian extension of $F$ depending on $K$. 

\subsection{Local systems on $\rS_K(\BC)$}
Fix one of the locally symmetric spaces $S_{\Gamma_i}$ arising as in~\eqref{locsymspaces}, and call it $\rS_{\Gamma}$. Its fundamental group is identified with $\Gamma$. Hence, any finite dimensional representation of $\rG$, by restriction to $\Gamma$, naturally defines a local system on $\rS_\Gamma$. For any irreducible representation $V_{\lambda}$ of $\rG$ of highest weight $\lambda$, we denote by the same symbol the resulting local system. One gets an isomorphism
\begin{equation}\label{eq:gpls}
H^\bullet(\Gamma, V_{\lambda}) \cong H^\bullet(\rS_\Gamma, V_{\lambda})\,.
\end{equation}
The above local systems provide local systems $V_{\lambda}$ on $\rS_K(\BC)$ in the obvious way.

\subsection{Hodge structures with coefficients and variations of Hodge structures on $\rS_K(\BC)$} \label{subsec_hodge}
For the notions and facts that we will use about pure and mixed Hodge structures and variations of Hodge structures, we refer the reader to~\cite{PS08}. Here, we mostly set up our notations. 

We will denote by $\MHS$ the category of mixed $\BQ$-Hodge structures. Moreover, we will need to work with a generalization of this category: for $\BK$ a finite extension of $\BQ$, the category $\MHS_{\BK}$ of mixed Hodge structures \emph{with coefficients in} $\BK$, or of mixed $\BK$-Hodge structures, is the pseudo-abelian completion of the category whose objects are the same as $\MHS$ and where the morphisms between objects $A,B$ are defined by $\Hom_{\MHS}(A,B) \otimes \BK$. Letting the number field $\BK$ vary, one obtains a category of Hodge structures with $\overline{\BQ}$-coefficients. There are obvious notions of weights and types for a $\overline{\BQ}$-Hodge structure. Notice that the types of a pure $\overline{\BQ}$-Hodge structure do not necessarily satisfy Hodge symmetry.

For any irreducible representation $V_{\lambda}$ of $\rG$ of highest weight $\lambda$, the morphism $h$ of~\eqref{Shim_datum}, defining the pure Shimura datum on $\rG$, induces a pure $F$-Hodge structure on $V_{\lambda}$. Conjugate $h$ via $g \in \rG(\BR)$ to obtain a morphism $h^g$ whose image lands in the real points $\rT(\BR)$ of the standard maximal torus of $\rG$. Then, by definition, the weight of the Hodge structure on $V_{\lambda}$ is the unique integer $w$ such that for any $t \in \BR$, 
\begin{equation} \label{weightrep}
\lambda \circ h^{g} (t) = t^{-w}\,.
\end{equation}
\begin{rmk} \label{hodge_paramchar} \rm
A computation shows that if $\lambda=(k_1,k_2,c,r)$, then 
\begin{equation} \label{r=w}
w=r\,.
\end{equation}
This property is the reason we chose to parametrize characters as in~\eqref{paramchar}. In particular, $V_k$ is endowed with a pure $F$-Hodge structure of weight $k$.
\end{rmk}

By the above considerations and by the theory of Shimura varieties, the local systems $V_{\lambda}$ of the previous paragraph are endowed with the structure of a variation of $F$-Hodge structure. 
\begin{rmk} \label{univ_abvar_rep} \rm
Remember that $V_k$ is a choice of lift to $\rG$ of the $\rG^{\der}$-representation $\Sym^k V$. Our choice is the one making the variation of Hodge structure $V_k$ isomorphic to the $k$-symmetric power of the relative $H^1$ of the universal abelian threefold over $\rS_K$ (Remark~\ref{univ_abvar}). 
\end{rmk}

There exists a sheafified version of mixed Hodge theory, including the theory of variations of Hodge structures as a special case. The sheaf-theoretic category that one considers is then called the category of \emph{mixed Hodge modules}. They are equipped with a \emph{six functor formalism} (in the sense studied in~\cite{DG22}; for example, one can pullback and pushforward mixed Hodge modules along arbitrary morphisms). We will only need to use this formalism as a black box. In particular, it implies that the cohomology groups $H^\bullet(\rS_K, V_{\lambda})$ are endowed with a canonical mixed $F$-Hodge structure.

\subsection{Cohomological dimension} \label{cohdim}
We can estimate the vanishing of the cohomology groups considered in the previous subsection. In fact, following~\cite{BoSe73}, we know that, since $\rG^{\der}$ is a semisimple group over $\BQ$, for any torsion-free arithmetic subgroup $\Gamma$ of $\rG^{\der}(\BQ)$ its cohomological dimension satisfies
$$\mr{cd} \Gamma= \dim \rG^{\der}(\BR)- \dim K_{\infty} -\mr{rank}_\BQ \rG^{\der},$$

In our case, we have $\dim \rG^{\der}(\BR)=8$, $\dim K_{\infty}= \dim \mathrm{U}(2)=4$ and $\mr{rank}_\BQ \rG^{\der}=1$, so that $\mr{cd} \Gamma = 3$. Therefore, for any irreducible representation $V_{\lambda}$ of $\rG$, we get $H^q( \rS_\Gamma, V_{\lambda})=0$ for all $q \geq 4$.

\section{Algebraic Hecke characters and their motives} \label{Heckechar}
In this section, we define algebraic Hecke characters, survey the properties of their $L$-functions, and recall basic facts about the associated motives. We then discuss the Bloch-Beilinson conjectures for Hecke characters, and explain how they imply the existence of non-trivial extensions of the unit Hodge structure by a twist of the Hodge structure associated to an algebraic Hecke character.

\subsection{Algebraic Hecke characters} \label{algHecke}
A Hecke character on a $\BQ$-torus $\mathcal{T}$ is a continuous homomorphism 
\[
\mathcal{T}(\BQ) \backslash \mathcal{T}(\BA) \rightarrow \BC^{\times}\,.
\]
We will often refer to a Hecke character on $\Res_{F \vert \BQ} \BG_{m, F}$ as a Hecke character of $F$, which in this case is simply a continuous homomorphism  
\[
F^{\times} \backslash \mathbb{I}_{F} \rightarrow \BC^{\times}.
\]
A Hecke character $\phi$ on $\mathcal{T}$ is \emph{algebraic} if there exists a character 
\[
\chi : \mathcal{T} \rightarrow \BG_m
\]
such that the component at infinity $\phi_{\infty}$ of $\phi$ is equal to the character on $\mathcal{T}(\BC)$ induced by $\chi^{-1}$. In this case, we say that $\phi$ is of \emph{type} $\chi$. When $\mathcal{T} = \Res_{F \vert \BQ} \BG_{m,F}$, $\phi$ being algebraic amounts to saying that there exist integers $(a,b)$ such that $\phi_{\infty}$ coincides with
\[
\begin{array}{ccc}
\BC^{\times} & \rightarrow & \BC^{\times} \\
z & \mapsto & z^{-a} \bar{z}^{-b} 
\end{array}
\]
In this case, we say that $(a,b)$ is the \emph{infinity type} of $\phi$, and $w(\phi):=a+b$ is its \emph{weight}. 
\begin{rmk} \rm
The notion of \emph{weight} can be defined for Hecke characters over any torus, in such a way that a Hecke character $\phi$ on a torus $\mathcal{T}$ is of weight $w(\phi)$ if and only if, denoting by $\bar{\phi}$ the Hecke character obtained by composing $\phi$ with complex conjugation in $\BC$,
\begin{equation} \label{phiconj_weight}
\phi \cdot \bar{\phi} = \vert \cdot \vert_{\BI_F}^{-w(\phi)}.
\end{equation}

 Now suppose that $\phi$ is a Hecke character of $F$ and denote by $\phi_{\BQ}$ its restriction to $\BI_{\BQ}$. From the above discussion and from the fact that the restriction to $\BR^{\times}$ of $\vert \cdot \vert_{\BI_F, \infty}$ is equal to the square of $\vert \cdot \vert_{\BI_\BQ, \infty}$, it is clear that 
\[
w(\phi_{\BQ})=2w(\phi).
\]
\end{rmk}

Note that the finite part $\phi_f$ of an algebraic Hecke character takes values in $\overline{\BQ}^{\times}$, see, e.g.~\cite[Sect.~2.5]{Har87a}. More precisely, it takes values in a number field $F^\prime$, called its \emph{field of values}.

Recall the maximal torus $\rT^\prime$ of $\rG^{\der}$, as introduced in Subsection~\ref{sl2}. We will mostly work with algebraic Hecke characters on $\T^\prime$. They will be identified with Hecke characters of $F$ through the isomorphism~\eqref{maxtorus_der}. As at the end of Subsection~\ref{irreducible}, consider the natural embedding $\rT^\prime \rightarrow \Res_{F \vert \BQ} \rT^\prime_F$. Fix a character $\chi : \T^{\prime}_F \rightarrow \BG_{m,F}$, and write it $\chi=(\kappa_1,\kappa_2)$ in the parameterization of Subsection~\ref{roots}. Then $\chi$ restricts, under the previous embedding, to the map given on $\BR$-points by 
\[
\left(
\begin{array}{ccc}
\bar{z} & & \\
& \bar{z}^{-1} z & \\
& & z^{-1}
\end{array}
\right) \mapsto z^{\kappa_2}\bar{z}^{\kappa_1-\kappa_2}\,.
\]
With these conventions, the weight of a Hecke character $\phi$ of $F$ of type $\chi=(\kappa_1,\kappa_2)$ is 
\[
w(\phi)=\kappa_1 
\]
and its infinity type is $(\kappa_2,\kappa_1-\kappa_2)$.

For later use, let us recall some basic properties of the $L$-function $L(\phi,s)$ of an algebraic Hecke character $\phi$ of $F$ (see, e.g.,~\cite[Ch.~0,~Sect.~6]{Scha88}). The region in which $L(\phi,s)$ is expressed as an absolutely convergent Euler product is the half-plane $\mbox{Re} \ s>1+\frac{w(\phi)}{2}$. In this region, the $L$-function is therefore holomorphic, and moreover known to be non-zero. Multiplying the Euler product by appropriate $\Gamma$-factors, we obtain a \emph{completed L-function} $\Lambda(\phi,s)$ which admits a meromorphic continuation to the whole of $\BC$. The function $\Lambda(\phi,s)$ and the meromorphically continued $L(\phi,s)$ are entire unless $\phi( \cdot ) = \vert \cdot \vert_{\BI_F}^t$ for some $t \in \BC$; in the latter case, $\Lambda(\phi,s)$ has first-order poles precisely in $s=-t$ and $s=1-t$, and $L(\phi,s)$ has a unique first-order pole in $s=1-t$. Finally, there exists a holomorphic function $\epsilon(s)$ such that $\Lambda(\phi,s)$ satisfies the functional equation
\begin{equation}\label{funceq}
\Lambda(\phi,s)=\epsilon(s)\Lambda(\phi^{-1},1-s)\,.
\end{equation}
Using~\eqref{phiconj_weight}, the above functional equation can be rewritten as
\begin{equation}\label{funceq_conj}
\Lambda(\phi,s)=\epsilon(s)\Lambda(\bar{\phi},w(\phi)+1-s)\,.
\end{equation}
\begin{rmk} \label{selfadjoint} \rm
For a Hecke character $\phi$ of $F$, define the Hecke character $\phi^{\perp}$ by putting
\begin{equation} \label{phi_perp}
\phi^{\perp}(x):=(\phi(\alpha(x)))^{-1}, \ \ \ x \in \BI_F\,.
\end{equation}
If $\phi^{\perp}=\phi \vert \cdot \vert^{w(\phi)}_{\BI_F}$, then 
\begin{enumerate}
    \item we have $\Lambda(\phi,s)=\Lambda(\bar{\phi},s)$, and the functional equation becomes
\[
\Lambda(\phi,s)=\epsilon(\phi)\Lambda(\phi,w(\phi)+1-s)\
\]
    \item $\epsilon(\frac{w+1}{2})$ is equal to $1$ or $-1$; we call this value the \emph{sign of the functional equation of} $\phi$ and denote it by $\epsilon(\phi)$.
    \end{enumerate}    
Now suppose that $w:=w(\phi)$ is an odd integer. If $\phi^{\perp}=\phi \vert \cdot \vert^{w}_{\BI_F}$ and $\epsilon(\phi)=-1$, then an analysis of the $\Gamma$-factors shows that $L(\phi,s)$ has to vanish at the \emph{central point} $s=\frac{w+1}{2}$. 
\end{rmk}
To an algebraic Hecke character of $F$ with field of values $F^\prime$, one associates a pure Hodge structure $H_{\phi}$ with coefficients in $F^\prime$, and a compatible system of $\ell$-adic Galois representations $(M_{\phi,\ell})_{\ell}$, for $\ell$ varying over the primes. The Hodge structure $H_{\phi}$ is of rank 1 over $F^\prime$. Its weight is equal to $w(\phi)$, and its Hodge type is the infinity type of $\phi$.

\subsection{Motives} \label{motives} In order to discuss motives attached to Hecke characters, let us quickly review the notions that we will use about motives. For more details, the reader can consult~\cite{And04} and~\cite{MNP13}. 

Given a smooth, projective variety $X$ over a field $E$, denote by $\CH^n(X)$ the Chow group of codimension $n$ algebraic cycles on $X$ modulo rational equivalence, with coefficients in $\BQ$. If $\dim_X$ is the dimension of $X$, then $\CH^{\dim_X}(X \times X)$ becomes a ring under a composition law defined using intersection of cycles. Its elements are called (classes of) degree 0 self-correspondences on $X$, where, more generally, a correspondence between smooth projective varieties $X$ and $Y$ is an algebraic cycle on $X \times Y$.

The category $\CHM(E)$ of \emph{Chow motives} with rational coefficients over $E$ has objects of the form $(X,p,m)$ where $X$ is a smooth projective variety over $E$, $p$ is an idempotent element of $\CH^{\dim_X}(X \times X)$, and $m$ is a integer. We often write $(X,p)$ for $(X,p,0)$. Morphisms are defined using rational equivalence classes of correspondences between smooth projective varieties, with maps $f:X \rightarrow Y$ between smooth projective varieties being interpreted as correspondences by identifying them with the \emph{transpose} of their graph $\Gamma_f$. The category $\CHM(E)$ is tensor and pseudo-abelian. Any Weil cohomology theory $H^{\bullet}( - )$  on the category $SmProj(E)$ of smooth projective varieties, towards graded vector spaces over a field $\mathbb{K}$ of characteristic zero, factors uniquely through a canonical functor 
\[
SmProj(E) \rightarrow \CHM(E)
\]
associating to a variety $X$ its Chow motive $M(X):=(X, \Gamma_{\id_X})$. If there exists an embedding $\sigma$ of $E$ into $\mathbb{C}$, then one can take $\mathbb{K}=\mathbb{Q}$ and choose $H^{\bullet}(-)$ as the functor associating to $X$ the singular cohomology of the analytification of $X$. Then, the resulting functor from $ \CHM(E)$ to graded $\BQ$-vector spaces is called \emph{Betti realization} and factors through the category of graded, pure $\mathbb{Q}$-Hodge structures. It sends $(X,p,m)$ to $pH^{\bullet}(X)(m)$. Here, $pH^{\bullet}(X)$ denotes the direct summand of $H^{\bullet}(X)$ obtained as the image of the idempotent correspondence $p$ acting on $H^{\bullet}(X)$, and $(-)(m)$ denotes the $m$th Tate twist in the Hodge-theoretical sense. 
If $E$ is of characteristic zero, then there exists a \emph{de Rham} realization with coefficients in $E$. 
In general, one can fix a prime $\ell$ different from the characteristic of $E$, take $\mathbb{K}= \BQ_{\ell}$ and choose $H^{\bullet}(-)$ as being $\ell$-adic étale cohomology. The resulting $\ell$-adic realization $M_{\ell}$ of a Chow motive $M$ is then endowed with the structure of a $\ell$-adic Galois representation, which allows one to define an $L$-function $L(M,s)$. 

To deal with more general varieties than smooth and projective ones, one enlarges $\CHM(E)$ to a symmetric monoidal, $\BQ$-linear triangulated category of \emph{constructible motives} over $E$, denoted $\DM_c(E)$, equipped with a fully faithful embedding
\[
\CHM(E) \hookrightarrow \DM_c(E)\,.
\]
The category $\DM_c(E)$ was first constructed by Voevodsky (who called it the category of geometrical motives and denoted it by $\DM_{gm}(E)$) starting from the complexes over the additive category of \emph{smooth} varieties equipped with \emph{finite correspondences}, with \emph{no} equivalence relation imposed on them. Any variety $X$ over $E$, even if non-projective or singular, has an associated motive $M(X)$ in $\DM_c(E)$, identified with its Chow motive when $X$ is smooth and projective. When $E$ embeds into $\BC$, the Betti realization extends to $\DM_c(E)$ and factors through a \emph{Hodge realization} functor $R_H$ towards the category $\MHS_{\BQ}$ of mixed $\BQ$-Hodge structures. 

For any field $\BK$, there are $\BK$-linear analogues $\CHM(E)_{\BK}$, $\DM_c(E)_{\BK}$, called \emph{with coefficients in} $\BK$, of the above categories, with a Hodge realization $R_H$ factoring through $\MHS_{\BK}$. An object $A$ of $\MHS_{\overline{\BQ}}$ is said to be of \emph{geometric origin} if $A$ belongs to the full Tannakian subcategory of $\MHS_{\overline{\BQ}}$ generated by the cohomology spaces of those objects which are in the image of $R_H$.

\subsection{Motives for Hecke characters}
By~\cite[Prop. 3.5]{DM91}, there exists a Chow motive $M_{\phi}$ with coefficients in a finite extension of $F^\prime$, whose realizations are the Hodge structure $H_{\phi}$ and the compatible system $(M_{\phi,\ell})_{\ell}$ (the construction runs as in~\cite[Ch.1, Thm. 4.1]{Scha88}, where one obtains an object in the weaker category of \emph{absolute Hodge motives}). Then, one has $L(\phi,s)=L(M_{\phi},s)$. To describe this motive, recall that by~\cite{DM91}, for every abelian variety $A$, there exists a canonical Chow motive $h^i(A)$ realising the $i$-th cohomology of $A$ in any Weil cohomology theory. Then, the motives $M_{\phi}$ have the following form: if the weight $w$ of $\phi$ is positive, there exist an abelian variety $A_{\phi}$ over $F$, with complex multiplication by a commutative semisimple algebra $T$ over $F$, and an idempotent element $e_{\phi}$ of $T^{\otimes w}$, such that 
\[
M_{\phi}=(h^1(A_{\phi})^{\otimes w},e_{\phi})\,.
\]

For any embedding $F \hookrightarrow \BC$, we can consider the singular cohomology spaces $H^{\bullet}(A_{\phi})$, with coefficients in a suitable number field, of the analytification of $A_{\phi}$ with respect to the chosen embedding. The space $H^1(A_{\phi})^{\otimes w}$ has a canonical pure Hodge structure of weight $w$, with coefficients in the same number field. The Hodge structure $H_{\phi}$ can then be described as
\begin{equation} \label{motivicH_phi}
H_{\phi} = e_{\phi} H^1(A_{\phi})^{\otimes w}\,.
\end{equation} 
By the Künneth theorem, $H^1(A_{\phi})^{\otimes w}$ is a direct factor of $H^w(A_{\phi}^w)$. The results of~\cite{DM91} imply that there exists an idempotent correspondence $e_{1,w}$ cutting out a direct factor of $h^w(A_{\phi}^w)$ realizing to $H^1(A_{\phi})^{\otimes w}$, so that we can also write
\[
M_{\phi}=(h^w(A^w_{\phi}), e_{\phi}e_{1,w})
\]
and the Hodge structure $H_{\phi}$ can also be described as
\[
H_{\phi} = e_{\phi} e_{1,w} H^w(A^w_{\phi})\,.
\]

If the weight of $\phi$ is negative, then the above descriptions hold for the \emph{duals} of $M_{\phi}$ and $H_{\phi}$.

\subsection{Bloch-Beilinson conjectures for Hecke characters} \label{BBHecke_expl}
Let $\phi$ be an algebraic Hecke character of $F$, of positive odd weight $w$. When $L(\phi,s)$ vanishes at the central point $s=\frac{w+1}{2}$, the Bloch-Beilinson conjectures predict the value, up to an algebraic number, of the first non-trivial Taylor coefficient of $L(\phi,s)$ at $s=\frac{w+1}{2}$. Let us formulate their prediction precisely (following~\cite{Nek94} and~\cite{Scho94}).

Consider the motive $M_{\phi}=(h^1(A_{\phi})^{\otimes w},e_{\phi})$ attached to $\phi$ as in the previous subsection. Recall that since the integer $\frac{w+1}{2}$ is critical for $M_{\phi}$, Deligne has used the comparison isomorphism between the Betti and de Rham realizations of the weight $-1$ motive $M_{\phi}(\frac{w+1}{2})$ to attach to it a canonical non-zero period $c^+(M_{\phi}) \in \BR$. He conjectured that if the value $L(M_{\phi},\frac{w+1}{2})$ does not vanish, then it should be equal to an algebraic multiple of $c^+(M_{\phi})$. He offered no statement when $L(\phi,s)$ has a zero at the central point. To express what is believed to happen in the latter case, consider, for a smooth, projective algebraic variety $X$ and for $n \in \{0, \dots, \dim X \}$, the subgroup $\CH^n(X)_0$ of the Chow group $\CH^n(X)$ formed by cycles homologically equivalent to 0. Bloch and Beilinson conjectured the existence of a canonical, non-degenerate \emph{height pairing}
\[
h^n_X: \CH^n(X)_0 \otimes \BQ \times \CH^{\dim X+1-n}(X)_0 \otimes \BQ \rightarrow \BR\,.
\]
Using it, they formulated a conjecture, whose statement we want to express in the special case of $L(\phi,s)$. 

To do so, we have to extend the conjecture to the Chow groups of our motive $M_{\phi}$. Recall that for $X,n$ as above, there is a natural action of $\CH^{\dim X}(X \times X)$ on $\CH^n(X)$, preserving $\CH^n(X)_0$. Denote by $\Gamma_*$ the action of an element $\Gamma \in \CH^{\dim X}(X \times X)$. For such a $\Gamma$, the pairing $h^n_X$ restricts to a pairing, also conjecturally non-degenerate
\[
h^n_{X,\Gamma}: \Gamma_* \CH^n(X)_0 \otimes \BQ \times \Gamma_* \CH^{\dim X+1-n}(X)_0 \otimes \BQ \rightarrow \BR\,.
\]
Then, writing $L^*(\phi,\frac{w+1}{2})$ for the first non-trivial coefficient of $L(\phi,s)$ at $s=\frac{w+1}{2}$, the conjecture reads as follows.

\begin{conjecture} \label{BBHecke}
Let $M_{\phi}=(h^w(A^w_{\phi}),e_{\phi}e_{1,w})$ be the Chow motive attached to an algebraic Hecke character $\phi$ of $F$ of positive odd weight $w$. Then 
\begin{enumerate}
\item $\ord_{s=\frac{w+1}{2}} L(\phi,s)= \dim_{\BQ} (e_{\phi} e_{1,w})_* \CH^{\frac{w+1}{2}}(A^{w}_{\phi})_0 \otimes \BQ$\,,
\item 
denoting by $\sim_{\overline{\BQ}}$ the relation of coinciding up to multiplication by a non-zero algebraic number,
\[
L^*\left(\phi,\frac{w+1}{2}\right) \sim_{\overline{\BQ}} c^+(M_{\phi}) \det\left(h^{\frac{w+1}{2}}_{A_{\phi}^w,e_{\phi}e_{1,w}}\right) \,.
\]
\end{enumerate}
\end{conjecture}
The conjecture extends in the obvious way to the case of negative weight $w$. 

Now, fix an element $Z \in \CH^{\frac{w+1}{2}}(A^{w}_{\phi})_0$. By a completely analogous construction to the one explained in the Introduction before stating Conjecture~\ref{Conjecture} (see e.g.~\cite[Par.~(4.2)]{Nek94}), such a cycle gives rise to an extension of the trivial Hodge structure by $H^w(A^w_{\phi})(\frac{w+1}{2})$. One gets an \emph{Abel-Jacobi map}
\[
\CH^{\frac{w+1}{2}} (A^w_{\phi})_0 \rightarrow \Ext^1_{\MHS_{\overline{\BQ}}}(\one, H^w(A^w_{\phi})(\frac{w+1}{2}))
\]
\begin{rmk}$\,$\rm \label{AJ}
The name \emph{Abel-Jacobi} comes from the fact that when considering mixed Hodge structures with \emph{integral} coefficients, the target of the above map, by Carlson's computation of extension spaces of mixed Hodge structures, can be identified with a complex torus (see e.g.~\cite[Thm.~3.31]{PS08}). If we replace $A^w_{\phi}$ by an elliptic curve $E$ over $\BQ$, the space $H^w(A^w_{\phi})(\frac{w+1}{2})$ by $H^1(E)(1)$, and the cycle $Z$ by the divisor $P - O$ with $P \in E(\BQ)$, then we find ourselves in the situation discussed in the Introduction before Conjecture~\ref{Conjecture}. In this case, the complex torus in question is identified with $E(\BC)$ itself, in such a way that the map $P \mapsto E_P$ can be expressed as $P \mapsto P - O$. More generally, if one replaces the elliptic curve by a higher genus smooth projective curve $C$, and $P - O$ by a degree 0 divisor on $C$, the complex torus is given by the complex points $J(C)(\BC)$ of the Jacobian $J(C)$ of $C$, and the map sending a divisor to the corresponding extension is induced by the classical Abel-Jacobi map
\[
\Div^0(C) \rightarrow J(C)(\BC)\,.
\]
\end{rmk}
Passing to the subspaces cut out by the idempotent $e_{\phi}e_{1,w}$, we get
\[
(e_{\phi}e_{1,w})_* \CH^{\frac{w+1}{2}} (A^w_{\phi})_0 \rightarrow \Ext^1_{\MHS_{\overline{\BQ}}}(\one, e_{\phi}e_{1,w}H^w(A^w_{\phi})(\frac{w+1}{2}))\,.
\]
Since $A_{\phi}^w$ is defined over a number field, the Abel-Jacobi map is conjectured to be \emph{injective}~(\cite[Sect.~5.6]{Bei87}). Hence, the same should hold for the map above, which, using the description~\eqref{motivicH_phi} of $H_{\phi}$, takes the form
\begin{equation} \label{AbJa}
(e_{\phi}e_{1,w})_* \CH^{\frac{w+1}{2}} (A^w_{\phi})_0 \rightarrow \Ext^1_{\MHS_{\overline{\BQ}}}(\one, H_{\phi}(\frac{w+1}{2}))\,.
\end{equation}
Now recall the character $\phi^{\perp}$ defined in~\eqref{phi_perp}, and suppose 
\[
\phi^{\perp}=\phi \vert \cdot \vert^w_{\BI_F}\,.
\]
Then, by Remark~\ref{selfadjoint}, there is a well-defined sign $\epsilon(\phi)$ of the functional equation of $\phi$, and if $\epsilon(\phi)=-1$ then $L(\phi,s)$ vanishes at the central point. Hence, the conjectural injectivity of the map~\eqref{AbJa} and point (1) of Conjecture~\ref{BBHecke} lead to the following (considerably weaker and less precise) conjecture. In its statement, we use the terminology \emph{of geometric origin} introduced at the end of Subsection~\ref{motives}. 

\begin{conjecture} \label{BBHecke_Hodge}
Let $\phi$ be an algebraic Hecke character of $F$ of odd weight w, such that $\phi^{\perp}=\phi \vert \cdot \vert^w_{\BI_F}$. If $\epsilon(\phi)=-1$, then there exists a non-trivial extension
\[
E_{\phi} \in \Ext^1_{\MHS_{\overline{\BQ}}}(\one, H_{\phi}(\frac{w+1}{2}))
\]
of geometric origin.
\end{conjecture}

\subsection{Motives of Picard surfaces}
The main aim of this paper is to construct, for a Hecke character $\phi$ verifying the required hypotheses, candidates for the extensions predicted by Conjecture~\ref{BBHecke_Hodge}. The construction will use the geometry of the Picard surfaces $\rS_K$ introduced in Subsection~\ref{sideremarks}. 

More precisely, consider the universal abelian variety $A_K$ over $\rS_K$ (Remark~\ref{univ_abvar}). For any integer $N$, we can consider its $N$-fold fibered power $A_K^N$ over $S_K$, and the motive $M(A^N)$ of its total space. It is an object of $\DM_c(F)$. 

Fix a highest weight $\lambda$ of $\rG_F$ and adopt the notations of Subsection~\ref{subsec_hodge}. By the main result of~\cite{Anc15} (see~\cite[before Rmk.~3.7]{Wil15} for a discussion in the case of Picard surfaces), there exist an integer $N$, and an idempotent endomorphism $e_{\lambda}$ of $M(A^N)$, such that the Hodge realization of the motive 
\begin{equation} \label{motive_lambda}
e_{\lambda} A_K^N
\end{equation}
coincides with $H^{\bullet}(\rS_K,V_{\lambda})$. This motive will be the source of the extensions \emph{of geometric origin} that we want to exhibit. 


\section{Boundary cohomology}\label{se:nerve}

In this section, we fix one of the Picard modular surfaces $\rS_K$ introduced in Section~\ref{sideremarks}, and let $\rS_{\Gamma}$ be a connected component of $\rS_K(\BC)$. We aim to describe the boundary cohomology of $\rS_{K}(\BC)$ with coefficients in local systems coming from representations of the underlying group $\rG$, together with their mixed Hodge structures and $\rG(\BA_f)$-actions. To do so, we will consider two different compactifications, the Baily-Borel compactification $\rS^*_{\Gamma}$ (a complex projective variety) and the Borel-Serre compactification $\overline{\rS}_{\Gamma}$ (a real analytic manifold with corners). 

\subsection{Boundary of the Baily-Borel compactification}
For facts recalled in this paragraph, we refer to~\cite[Sects.~3.6-3.7]{Pin92}. Any $\BQ$-conjugacy class of parabolics in $\rG$ provides a set of strata of the boundary of the Baily-Borel compactification of $\rS_{\Gamma}$. More precisely, given a representative $\rP$ of such a conjugacy class, there is a canonical reductive subgroup $\rM$ of the Levi component of $\rP$, and a canonical conjugacy class of morphisms
\begin{equation} \label{shimdatum_boundary}
h_{\rM} : \BS \rightarrow \rM_{\BR}
\end{equation}
providing a Shimura datum attached to $\rM$. Each of the strata of $\rS_{\Gamma}^*$ corresponding to $\rP$ is then a connected component of a Shimura variety associated with this Shimura datum. 

In our case, the only $\BQ$-conjugacy class of parabolics in $\rG$ is the one of the Borel $\rB$. The Levi component of $\rB$ is the torus $\rT$, and hence, the positive dimensional, reductive subgroup $\rM$ is a torus as well. Thus, the Baily-Borel boundary is a disjoint union of 0-dimensional strata, called \emph{cusps}.

\subsection{Boundary of the Borel-Serre compactification}\label{BSbound}
For facts recalled in this paragraph, we refer to~\cite[pp. 567-570]{Har87b}. The Borel-Serre compactification $\overline{\rS}_{\Gamma}$ of $\rS_{\Gamma}$ admits a canonical projection $\pi$ to the Baily-Borel one. Let us describe, in our situation, the fiber of $\pi$ over each cusp $p$ of $\rS^*_{\Gamma}$. To do so, let us write 
\[
\mathrm{U}^{\prime}:=\mathrm{U} \cap \rG^{\der}
\]
Denote by $A_{\rG}$ the identity component of the real points of the maximal $\BQ$-split subtorus $\rT^s$ of $\rG^{\der}$. From~\eqref{maxtorus_der}, we see that there is an isomorphism of Lie groups $\rT^\prime (\BR) \simeq \BC^{\times}$ which by~\eqref{maxtorus_Qsplit} restricts to $A_{\rG} \simeq \BR^{\times}_{>0}.$

Define $^0 T$ as the Lie group $A_{\rG} \backslash \rT^\prime (\BR)$. It is isomorphic to $S^1$. Denote by $K^{T}_{\infty}$ the intersection with $^0 T$ of a maximal compact subgroup of $\rG^{\der}(\BR)$. By conjugating the maximal compact subgroup described by~\eqref{maxcomp} via the change of basis matrix given in~\eqref{parab_to_diag}, one sees that $K^{T}_{\infty}$ coincides with $^0 T$ itself. On the other hand, denote $\Gamma_{\mathrm{U}}:= \Gamma \cap \mathrm{U}^\prime (\BR)$. The general theory tells us that the space $\pi^{-1}(p)$ has the structure of a fibration, with base a quotient of $^0 T \backslash K^T_{\infty}$ (hence a point in our case), and with fiber diffeomorphic to the nilmanifold
\[
\rS_{\mathrm{U}_{\Gamma}}:=\Gamma_{\mathrm{U}} \backslash \mathrm{U}^\prime (\BR)
\] 

This means that in our case, $\pi^{-1}(p) \simeq \rS_{\mathrm{U}_{\Gamma}}$. Now $\mathrm{U}^\prime$ is inserted into an exact sequence
\[
0 \rightarrow \mathrm{W} \rightarrow \mathrm{U}^\prime \rightarrow \widetilde{\mathrm{U}^\prime} \rightarrow 0
\]
where $\mathrm{W}$ is the commutator subgroup of $\mathrm{U}^\prime$, isomorphic to $\BG_a$, and $\widetilde{\mathrm{U}^\prime}$ is isomorphic to $\Res_{F \vert \BQ} \BG_{a,F}$. Since we are working with a torsion-free $\Gamma$, this yields a diffeomorphism of $\rS_{\mathrm{U}_{\Gamma}}$ with a $S^1$-fiber bundle over an elliptic curve.

\subsection{Boundary cohomology and its Hodge structure} \label{boundhodge}
Let $V_{\lambda}$ be the irreducible representation of $\rG$ of highest weight $\lambda$. Let $\widetilde{\rS_{\Gamma}}$ be a compactification of $\rS_{\Gamma}$ and let $j: \rS_{\Gamma} \hookrightarrow \widetilde{\rS_{\Gamma}}$ be the corresponding open immersion, with boundary $i: \partial \widetilde{\rS_{\Gamma}} \hookrightarrow \widetilde{\rS_{\Gamma}}$. Boundary cohomology of $\rS_{\Gamma}$ with coefficients in $V_{\lambda}$  is defined as hypercohomology of a complex of sheaves over $\partial \widetilde{\rS_{\Gamma}}$:
\[
\partial H^n(\rS_{\Gamma}, V_{\lambda}) := \mathbb{H}^n(\partial \widetilde{\rS_{\Gamma}}, i^* Rj_* V_{\lambda})
\]
where $Rj_*$ stands for the total derived direct image functor with respect to $j$. 

By proper base change, boundary cohomology does not depend on the choice of the compactification $\widetilde{\rS_{\Gamma}}$. There is a natural long exact sequence
\begin{equation} \label{longexloc}
\cdots \rightarrow H_c^n(\rS_{\Gamma}, V_{\lambda}) \rightarrow H^n(\rS_{\Gamma}, V_{\lambda}) \rightarrow \partial H^n(\rS_{\Gamma}, V_{\lambda}) \rightarrow \cdots
\end{equation}
and a canonical mixed Hodge structure on the spaces $\partial H^{\bullet}(\rS_{\Gamma}, V_{\lambda})$ making the above a long exact sequence of mixed Hodge structures. This can be seen, for example, by applying the formalism of mixed Hodge modules mentioned in Subsection~\ref{subsec_hodge}.

The following Proposition gives a complete description of the mixed Hodge structure on $\partial H^{\bullet}(\rS_{\Gamma}, V_{\lambda})$, using the spaces $I_{\mu,\nu}$ introduced in~\eqref{I_mu,nu}. It is equivalent to the main result of~\cite{Anc17}, but we have chosen to write it in detail using our parametrizations and notations and to provide an outline of its proof for the convenience of the reader. 

\begin{prop} \label{boundaryweights}
Let $\lambda=(k_1,k_2,c,r)$. Let $m$ be the number of cusps of $\rS_{\Gamma}$. Then, for any integer $n$, the mixed Hodge structures $\partial H^{n}(\rS_{\Gamma}, V_{\lambda})$ are semisimple, and described as follows:
\begin{enumerate}[wide, labelwidth=!, labelindent=0pt]
\item the space $\partial H^{n}(\rS_{\Gamma}, V_{\lambda})$ is trivial for $n$ outside the set $\{0, \dots, 3 \}$. 
\item For $n=0$, 
\[\partial H^{0}(\rS_{\Gamma}, V_{\lambda}) \cong I_{k_2,k_1-k_2}^{\oplus m}\] 
\[
\mbox{where} \ I_{k_2,k_1-k_2} \ \mbox{is endowed with a 1-dimensional pure $F$-Hodge structure of}
\]
\[
\mbox{\emph{weight}:} \ r-k_1 
\]
\[
\mbox{\emph{type}:} \ \left(-\frac{k_1+k_2+c-2r}{4}, -\frac{3k_1-k_2-c-2r}{4}\right)
\]
\item For $n=1$,
\[
\partial H^{1}(\rS_{\Gamma}, V_{\lambda}) \simeq I_{k_1+1,-k_1+k_2-2}^{\oplus m} \oplus I_{-k_2-2,k_1+1}^{\oplus m} 
\]
\[
\mbox{where} \ I_{k_1+1,-k_1+k_2-2} \ \mbox{is endowed with a 1-dimensional pure $F$-Hodge structure of}
\]
\[
\mbox{\emph{weight}:} \ (r+1)-k_2
\]
\[
\mbox{\emph{type}:} \ \left(-\frac{k_1+k_2+c-2r}{4}, -\frac{-k_1+3k_2-4-c-2r}{4}\right)
\]
\[
\mbox{and} \ I_{-k_2-2,k_1+1} \ \mbox{is endowed with a 1-dimensional pure $F$-Hodge structure of}
\]
\[
\mbox{\emph{weight}:} \ (r+1)-(k_1-k_2)
\]
\[
\mbox{\emph{type}:} \ \left(-\frac{k_1-3k_2-4+c-2r}{4}, -\frac{3k_1-k_2-c-2r}{4}\right)
\]
\item For $n=2$,
\[
\partial H^{2}(\rS_{\Gamma}, V_{\lambda}) \simeq I_{k_1-k_2,-k_1-3}^{\oplus m} \oplus I_{-k_1-3,k_2}^{\oplus m}
\]
\[
\mbox{where} \ I_{k_1-k_2,-k_1-3} \ \mbox{is endowed with a 1-dimensional pure $F$-Hodge structure of}
\]
\[
\mbox{\emph{weight}:} \ (r+2)+k_2+1
\]
\[
\mbox{\emph{type}:} \ \left(-\frac{k_1-3k_2-4+c-2r}{4}, -\frac{-k_1-k_2-8-c-2r}{4}\right)
\]
\[
\mbox{and} \ I_{-k_1-3,k_2} \ \mbox{is endowed with a 1-dimensional pure $F$-Hodge structure of}
\]
\[
\mbox{\emph{weight}:} \ (r+2)+(k_1-k_2)+1
\]
\[
\mbox{\emph{type}:} \ \left(-\frac{-3k_1+k_2-8+c-2r}{4}, -\frac{-k_1+3k_2-4-c-2r}{4}\right)
\]

\item 
For $n=3$
\[
\partial H^{3}(\rS_{\Gamma}, V_{\lambda}) \simeq I_{-k_1-4,-k_2-2}^{\oplus m}
\]
\[
\mbox{where} \ I_{-k_1-4,-k_2-2} \ \mbox{is endowed with a 1-dimensional pure $F$-Hodge structure of}
\]
\[
\mbox{\emph{weight}:} \ (r+3)+k_1+1 
\]
\[
\mbox{\emph{type}:} \ \left(-\frac{-3k_1+k_2-8+c-2r}{4}, -\frac{-k_1-k_2-8-c-2r}{4}\right)
\]
\end{enumerate}
\end{prop}
\begin{proof}
The assertion of point (1) follows from the vanishing discussed in Subsection~\ref{cohdim} and from the long exact sequence~\eqref{longexloc}. 

To compute boundary cohomology in degree $n \in \{0, \dots, 3 \}$, we choose to use the Baily-Borel compactification $\rS^*_{\Gamma}$. Then, since $\partial \rS^*_{\Gamma}$ is 0-dimensional, 
\[
\partial H^n(\rS_{\Gamma}, V_{\lambda}) \simeq i^* R^n j_* V_{\lambda}
\]
This is also an isomorphism of mixed Hodge structures, once we consider the mixed Hodge structure on $i^* R^n j_* V_{\lambda}$ given by the theory of mixed Hodge modules. Now, for any cusp $c \in \partial \rS^*_{\Gamma}$, write $i_c$ for the composition of $i$ with the closed immersion of $c$ into $\partial \rS^*_{\Gamma}$. We then get 
\[
\partial H^n(\rS_{\Gamma}, V_{\lambda}) \simeq \bigoplus\limits_{c \in \partial \rS^*_{\Gamma}} i_c^* R^n j_* V_{\lambda}
\]

Now consider the subgroup $\rM$ of $\rT$, and its associated Shimura datum $h_\rM$, appearing in~\eqref{shimdatum_boundary}. The general formalism of~\cite[Ch.~4]{Pin90} has been spelled out in this special case in~\cite[Lem.~3.6, Lem.~3.8]{Anc17}. The outcome is that the map $h_M$ is defined on $\BC$-points by 
\begin{equation} \label{shimdatum_boundarySU(2,1)}
(z_1,z_2) \mapsto ( \left(
\begin{array}{ccc}
z_1z_2 & 0 & 0 \\
0 & z_1 & 0 \\
0 & 0 & 1
\end{array}
\right), z_1z_2)\,
\end{equation}
Through $h_\rM$, any representation of $\rM_{\BR}$ acquires a semisimple Hodge structure. 

By~\cite[Thm.~2.9]{BW04}, there is an isomorphism of mixed Hodge structures
\[
i_c^* R^n j_* V_{\lambda} \simeq \bigoplus\limits_{p+q=n} H^p(\overline{H}_C,H^q( \mathrm{U}_F, V_{\lambda}))
\]
where $\overline{H}_C$ is a certain neat arithmetic subgroup of $\rT / \rM$, and the Hodge structure on the spaces $H^p(\overline{H}_C,H^q( \mathrm{U}_F, V_{\lambda}))$ is the one obtained through~\eqref{shimdatum_boundarySU(2,1)} by considering them as $\rM$-representations, via restriction of the natural $\rT$-representation. 

Since $\rT$ is isogenous to $\Res_{F \vert \BQ} \BG_{m,F} \times \Res_{F \vert \BQ} \BG_{m,F}$, any neat arithmetic subgroup of it is trivial, hence the same holds for neat arithmetic subgroups of its quotients. So, the above formula becomes 
\[
i_c^* R^n j_* V_{\lambda} \simeq H^n( \mathrm{U}_F, V_{\lambda})
\]
In particular, notice that the Hodge structure on $i_c^* R^n j_* V_{\lambda}$ is independent of the cusp $c$, so that if $m$ is the number of cusps, then as Hodge structures
\[
\partial H^n(\rS_{\Gamma}, V_{\lambda}) \simeq H^n(\mathrm{U}_F, V_{\lambda})^{\oplus m}
\]

By applying Kostant's theorem~\ref{KostThm}, we see that
\[
H^n(\mathrm{U}_F, V_{\lambda}) \simeq \bigoplus \limits_{w \in W^{\prime} \vert l(w)=n} V^{\rT_F}_{w.(\lambda+\delta)-\delta},
\]
where $V^{\rT_F}_{w.(\lambda+\delta)-\delta}$ is the one-dimensional representation on which the split torus $\rT_F$ acts via the character $w.(\lambda+\delta)-\delta$. Denoting the latter character by $\chi$, and expressing it in our parametrization by 
\[
\chi = (\kappa_1, \kappa_2, \gamma, \rho)
\]
a computation shows that through $\chi \circ h_\rM$, the element $(z_1,z_2) \in \BS(\BC)$ acts on $V^{\rT_F}_{\chi}$ via multiplication by
\begin{equation} \label{HT}
z_1^{\frac{\kappa_1+\kappa_2}{4}+\frac{1}{4}\gamma-\frac{1}{2}\rho} z_2^{\frac{3\kappa_1-\kappa_2}{4}-\frac{1}{4}\gamma-\frac{1}{2}\rho}\,.
\end{equation}
So, the \emph{negatives} of the exponents of $z_1$ and $z_2$ appearing above give us the type of the one-dimensional $F$-Hodge structure carried by $V^{\rT_F}_{\chi}$. 

On the other hand, using the isomorphism~\eqref{maxtorus_der}, an element $a \in F^{\times}$ acts on $V^{\rT_F}_{\chi}$ via multiplication by
\[
a^{\kappa_2} \alpha(a)^{\kappa_1-\kappa_2}
\]
so that we get an identification, as $\rT^\prime$-representations,
\[
V^{\rT_F}_{\chi} = I_{\kappa_2,\kappa_1-\kappa_2}
\]

By plugging in the explicit formulae, provided in Subsection~\ref{KostantRep}, for the different $\chi$'s appearing in Kostant's theorem, assertions (2)-(5) of the statement are then proven. 
\end{proof}

\begin{rmk}$\,$\rm
\begin{enumerate}
\item The assertions in Proposition~\ref{boundaryweights} concerning the weights of the Hodge structures of interest can also be found, with our same parametrizations, in~\cite[proof of Thm.~3.8]{Wil15}. In our way of writing the weights, we have imitated \emph{loc. cit.} in emphasizing the deviation from \emph{purity} (i.e., the difference, in a fixed degree $i$, between the actual weight and the \emph{pure} weight $r+i$). As \emph{op. cit.} shows, controlling this difference is important for being able to construct motives for cuspidal automorphic forms on $\rG$.
\item The \emph{topological} description of the boundary cohomology spaces $\partial H^{\bullet}(\rS_{\Gamma}, V_{\lambda})$, with no computation of their Hodge structure (since~\cite{BW04} was not available at that time) was already determined in~\cite[Eq.~(2.1.2)]{Har87b}. In the notation of \emph{loc.~cit.}, the Weyl elements $w_{\alpha}$, $w_{\bar{\alpha}}$, $w_{\bar{\alpha}} w_{\alpha}$, $w_{\alpha} w_{\bar{\alpha}}$ and $\theta$ correspond respectively to our elements $(2\ 3)$, $(1\ 2)$, $(1\ 2\ 3)$, $(1\ 3\ 2)$ and $(1\ 3)$. 
\item For his analysis of boundary cohomology in~\cite{Har87b}, Harder chose to work with the Borel-Serre compactification $\overline{\rS}_{\Gamma}$. Its crucial property is that the open immersion 
\[
\rS_{\Gamma} \hookrightarrow \overline{\rS}_{\Gamma}
\]
is an homotopy equivalence, so that the sheaf $V_{\lambda}$ canonically extends to a sheaf on $\overline{\rS}_{\Gamma}$, denoted by the same symbol. Hence, there is no need to compute hypercohomology: the spaces $\partial H^n(\rS_{\Gamma}, V_{\lambda})$ are canonically isomorphic to $H^n(\partial \overline{\rS}_{\Gamma}, V_{\lambda})$. As a consequence, for computing boundary cohomology associated with higher-rank groups underlying a Shimura datum, the Borel-Serre compactification has a strong computational advantage over the Baily-Borel one. In our situation, the boundary structure being as simple as possible, determining boundary cohomology as vector spaces with either choice is equally easy. The Baily-Borel compactification gives us a way to describe their canonical mixed Hodge structure. In the following section, we will recall how the Borel-Serre one provides a description of the $\rG(\BA_f)$-module structure of boundary cohomology.
\end{enumerate}
\end{rmk}

In the following two corollaries, we specialize Proposition~\ref{boundaryweights} to the case of $V_{\lambda}$ being one of the two local systems $V_k$ or $V_k^{\vee}$ defined in Subsection~\ref{irreducible}.
\begin{coro}\label{boundHS_Vk}
For $n \in \{0, \dots, 3\}$, the Hodge structure on the boundary cohomology spaces $\partial H^n(\rS_{\Gamma},V_k)$ is described as follows.
\begin{enumerate}
\item For $n=0$,
\[
\partial H^{0}(\rS_{\Gamma}, V_k) \cong I_{0,k}^{\oplus m} 
\]
\[
\mbox{where} \ I_{0,k} \ \mbox{is endowed with a 1-dimensional pure $F$-Hodge structure of}
\]
\[
\mbox{\emph{weight}:} \ 0 
\]
\[
\mbox{\emph{type}:} \ (0, 0)
\]
\item For $n=1$,
\[
\partial H^{1}(\rS_{\Gamma}, V_k) \simeq I_{k+1,-k-2}^{\oplus m} \oplus I_{-2,k+1}^{\oplus m} 
\]
\[
\mbox{where} \ I_{k+1,-k-2} \ \mbox{is endowed with a 1-dimensional pure $F$-Hodge structure of}
\]
\[
\mbox{\emph{weight}:} \ k+1
\]
\[
\mbox{\emph{type}:} \ (0,k+1)
\]
\[
\mbox{and} \ I_{-2,k+1} \ \mbox{is endowed with a 1-dimensional pure $F$-Hodge structure of}
\]
\[
\mbox{\emph{weight}:} \ 1
\]
\[
\mbox{\emph{type}:} \ (1,0)
\]
\item For $n=2$,
\[
\partial H^{2}(\rS_{\Gamma}, V_k) \simeq I_{k,-k-3}^{\oplus m} \oplus I_{-k-3,0}^{\oplus m}
\]
\[
\mbox{where} \ I_{k,-k-3} \ \mbox{is endowed with a 1-dimensional pure $F$-Hodge structure of}
\]
\[
\mbox{\emph{weight}:} \ (k+2)+1
\]
\[
\mbox{\emph{type}:} \ (1, k+2)
\]
\[
\mbox{and} \ I_{-k-3,0} \ \mbox{is endowed with a 1-dimensional pure $F$-Hodge structure of}
\]
\[
\mbox{\emph{weight}:} \ (k+2)+k+1
\]
\[
\mbox{\emph{type}:} \ (k+2,k+1)
\]

\item 
For $n=3$
\[
\partial H^{3}(\rS_{\Gamma}, V_k) \simeq I_{-k-4,-2}^{\oplus m}
\]
\[
\mbox{where} \ I_{-k-4,-2} \ \mbox{is endowed with a 1-dimensional pure $F$-Hodge structure of}
\]
\[
\mbox{\emph{weight}:} \ (k+3)+k+1 
\]
\[
\mbox{\emph{type}:} \ (k+2,k+2)
\]
\end{enumerate}
\end{coro}

Using Remark~\ref{conj_highestweight}, one sees:
\begin{coro} \label{boundHS_Vkdual}
For $n \in \{0, \dots, 3\}$, the Hodge structure on the boundary cohomology spaces $\partial H^n(\rS_{\Gamma},V^{\vee}_k)$ has:
\begin{itemize}
    \item as weights, the same as those of the Hodge structure on $\partial H^n(\rS_{\Gamma},V_k)$;
    \item as types, the conjugates of the ones of the Hodge structure on $\partial H^n(\rS_{\Gamma},V_k)$. 
\end{itemize}
\end{coro}


\subsection{Automorphic structure of boundary cohomology}\label{ss:autbound}

We end this section by describing the automorphic structure of boundary cohomology. For this purpose, we switch to the adelic setting and work with the spaces 
\begin{equation*}
\rS_K(\BC):={\rG}(\BQ) \backslash (X \times {\rG}(\BA_f) / K)
\end{equation*}
defined in Remark~\ref{sideremarks}. For any irreducible representation $V_{\lambda}$ of $\rG$ of highest weight $\lambda$, denoted by $V_{\lambda,K}$  the corresponding local system on $\rS_K(\BC)$. Letting $K$ vary among the compact open subgroups of ${\rG}(\BA_f)$, one obtains a projective system $(\rS_K(\BC),V_{\lambda, K})_K$ of spaces and sheaves, the cohomology of whose projective limit $(\rS(\BC),V_{\lambda})$ is such that
\begin{equation*}
H^{\bullet}(\rS(\BC),V_{\lambda})=\varinjlim_K H^{\bullet}(\rS_K(\BC),V_{\lambda,K})
\end{equation*}
and is endowed with a canonical structure of ${\rG}(\BA_f)$-module. There is an analogously defined cohomology with compact support $H^{\bullet}_c(\rS(\BC),V_{\lambda})$, and one defines \emph{interior cohomology}
\begin{equation} \label{intcoh}
H_!^{\bullet}(\rS(\BC), V_{\lambda})
\end{equation}
as the image of the natural morphism $H_c^{\bullet}(\rS(\BC), V_{\lambda}) \rightarrow H^{\bullet}(\rS(\BC), V_{\lambda})$. 
Interior cohomology at a finite level is recovered by taking $K$-invariants,
\begin{equation} \label{intcoh^K}
H_!^{\bullet}(\rS_K(\BC), V_{\lambda})=H_!^{\bullet}(\rS(\BC), V_{\lambda})^K
\end{equation}
\begin{rmk}$\,$\rm \label{Hecke_algebra}
The $\rG(\BA_f)$-action at infinite level equips the above space of $K$-invariants with a structure of module under the \emph{level} $K$ \emph{Hecke algebra}
\[
\mathcal{C}^{\infty}_c(\rG,K)
\]
i.e. the algebra of smooth functions with compact support on $\rG(\BA_f)$, bi-invariant under $K$, with product defined by convolution. 
\end{rmk}

Denote by $\partial {\rS}_K(\BC)$ the boundary of the Borel-Serre compactification at any level $K$. From this and from the restriction to $\partial {\rS}_K(\BC)$ of the canonically extended $V_{\lambda}$, we obtain an analogous projective system and, in the projective limit, a ${\rG}(\BA_f)$-module that we denote 
\[
\partial H^{\bullet}(\rS(\BC),V_{\lambda})\,.
\]
Again, boundary cohomology at a finite level is recovered by taking $K$-invariants, 
\[
\partial H^{\bullet}(\rS_K(\BC),V_{\lambda})=\partial H^{\bullet}(\rS(\BC),V_{\lambda})^K\,.
\]

To describe the \q{{\it{automorphic structure of boundary cohomology}}} means, in a first approximation, to describe the structure of the ${\rG}(\BA_f)$-module $\partial H^{\bullet}(\rS(\BC),V_{\lambda})$. This description will be given in terms of Hecke characters.  

\begin{defi}
Let $\phi$ be an algebraic Hecke character of $\rT$ and consider its finite component
\[
\phi_f : \rT(\BA_f) \rightarrow \overline{\BQ}^{\times}
\]
Let $\rT(\BA_f)$ act on $\overline{\BQ}$ via multiplication by $\phi_f$. We denote by 
\[
\overline{\BQ}_{\phi}
\]
the resulting $\rT(\BA_f)$-module. We make it a $\rB(\BA_f)$-module by considering the decomposition $\rB = \rT \mathrm{U}$ and letting $\mathrm{U}(\BA_f)$ act trivially. 
\end{defi}

Recall the \q{twisted} action $w \star \lambda$ of an element $w \in W$ of the Weyl group of $\rG$ on a weight $\lambda$, as in Subsection~\ref{KostantRep}. Then, the theorem we are interested in reads as follows:
\begin{thm}\cite[Thm.~1]{Har87b}\label{thm:GH1}
There exists a canonical isomorphism of ${\rG}(\BA_f)$-modules
\begin{equation}
H^{\bullet}(\partial \rS(\BC),V_{\lambda, \overline{\BQ}}) \simeq \bigoplus\limits_{w \in W} \bigoplus_{\mbox{\emph{\tiny{type}}}(\phi)= w \star \lambda} \Ind_{{\rB}(\BA_f)}^{{\rG}(\BA_f)} \overline{\BQ}_{\phi}
\end{equation}
\end{thm}

Let us now specialize the above theorem in the case of $V_{\lambda}$ being one of the two local systems $V_k$ or $V_k^{\vee}$ defined in Subsection~\ref{irreducible}. If the type of the restriction to $\rT^\prime$ of a Hecke character $\phi$ of $\rT$ is $\chi = (\kappa_1, \kappa_2)$, then we denote\footnote{The notation is motivated by the fact that, in this case, the infinity type of the restriction of $\phi$ to $\rT^\prime$ is $(\kappa_2,\kappa_1-\kappa_2)$ (cfr. Subsection~\ref{Heckechar}).}
by
\[
\overline{\BQ}_{\kappa_2,\kappa_1-\kappa_2}
\]
the $\rB(\BA_f)$-module $\overline{\BQ}_{\phi}$. The computation of the characters $w \star \lambda$ provided in Subsection~\ref{KostantRep} then yields the following. 

\begin{coro} \label{boundaut_Vk}
Let $\lambda=(k,0,k,k)$. Then there are canonical isomorphisms of ${\rG}(\BA_f)$-modules
\begin{align*}
& H^{0}(\partial \rS(\BC),V_{k, \overline{\BQ}}) \simeq  \bigoplus_{\mbox{\emph{\tiny{type}}}(\phi)=id \star \lambda} \Ind_{{\rB}(\BA_f)}^{{\rG}(\BA_f)} \overline{\BQ}_{0,k}\\
& H^{1}(\partial \rS(\BC),V_{k, \overline{\BQ}}) \simeq  \bigoplus_{\mbox{\emph{\tiny{type}}}(\phi)= (1\ 2) \star \lambda} \Ind_{{\rB}(\BA_f)}^{{\rG}(\BA_f)} \overline{\BQ}_{k+1,-k-2} \oplus \bigoplus_{\mbox{\emph{\tiny{type}}}(\phi)= (2\ 3) \star \lambda} \Ind_{{\rB}(\BA_f)}^{{\rG}(\BA_f)} \overline{\BQ}_{-2,k+1}\\
& H^{2}(\partial \rS(\BC),V_{k, \overline{\BQ}}) \simeq  \bigoplus_{\mbox{\emph{\tiny{type}}}(\phi)= (1\ 2\ 3) \star \lambda} \Ind_{{\rB}(\BA_f)}^{{\rG}(\BA_f)} \overline{\BQ}_{k,-k-3} \oplus \bigoplus_{\mbox{\emph{\tiny{type}}}(\phi)= (1\ 3\ 2) \star \lambda} \Ind_{{\rB}(\BA_f)}^{{\rG}(\BA_f)} \overline{\BQ}_{-k-3,0}\\
& H^{3}(\partial \rS(\BC),V_{k, \overline{\BQ}}) \simeq  \bigoplus_{\mbox{\emph{\tiny{type}}}(\phi)= (1\ 3) \star \lambda} \Ind_{{\rB}(\BA_f)}^{{\rG}(\BA_f)} \overline{\BQ}_{-k-4,-2}
\end{align*}  
\end{coro}

\begin{rmk}$\,$\rm
The relation between the above Corollary~\ref{boundaut_Vk} and the Corollary~\ref{boundHS_Vk} is the following (a similar relation holds indeed for any system of coefficients $V_{\lambda}$). 

Fix a neat compact open subgroup $K \subset \rG(\BA_f)$ and a connected component $\rS_{\Gamma}$ of $\rS_K(\BC)$. Then $H^{\bullet}(\rS_{\Gamma}, V_{k, \overline{\BQ}})$ and $\partial H^{\bullet}(\rS_{\Gamma}, V_{k, \overline{\BQ}})$ get identified respectively to a direct summand of $H^{\bullet}(\rS(\BC), V_{k, \overline{\BQ}})^K$ and $H^{\bullet}(\partial \rS(\BC), V_{k, \overline{\BQ}})^K$. 

Under the resulting inclusion
\[
\partial H^{\bullet}(\rS_{\Gamma}, V_{k, \overline{\BQ}}) \hookrightarrow H^{\bullet}(\partial \rS(\BC), V_{k, \overline{\BQ}})^K
\]
the spaces $I_{\mu,\nu,\overline{\BQ}}:=I_{\mu,\nu} \otimes_{F} \overline{\BQ}$ on the left-hand side become identified to one-dimensional subspaces of the spaces $(\Ind_{{\rB}(\BA_f)}^{{\rG}(\BA_f)} \overline{\BQ}_{\mu,\nu})^K$ on the right-hand side. Hence, if we write
\[
m_{\mu,\nu} := \dim_{\overline{\BQ}} (\Ind_{{\rB}(\BA_f)}^{{\rG}(\BA_f)} \overline{\BQ}_{\mu,\nu})^K
\]
then we get an isomorphism
\begin{equation} \label{HodgeInd}
I^{\oplus m_{\mu,\nu}}_{\mu,\nu, \overline{\BQ}} \simeq (\Ind_{{\rB}(\BA_f)}^{{\rG}(\BA_f)} \overline{\BQ}_{\mu,\nu})^K
\end{equation}
under which the spaces on the right hand side become endowed with a pure $\overline{\BQ}$-Hodge structure, direct sum of copies of the pure, 1-dimensional $\overline{\BQ}$-Hodge structure $I_{\mu,\nu, \overline{\BQ}}$.
\end{rmk}

\begin{coro}
Let $\lambda=(k,k,0,k)$. Then the $\rG(\BA_f)$-module structure of the boundary cohomology spaces $
H^{\bullet}(\partial \rS(\BC),V^{\vee}_{k, \overline{\BQ}})$ is the same as the one on $
H^{\bullet}(\partial \rS(\BC),V_{k, \overline{\BQ}})$, up to replacing the infinity types of the Hecke characters appearing in the latter with their \emph{conjugates}. 
\end{coro}


\section{Eisenstein cohomology}\label{se:eis}
In this section, we fix again one of the Picard modular surfaces $\rS_K$ introduced in Section~\ref{sideremarks} and we describe its Eisenstein cohomology by recalling the results of the article~\cite{Har87b} of Harder. This is used at the end of the section to construct the first piece of the extensions, which are the object of our Theorem~\ref{mainthm}. Namely, we construct a sub-Hodge structure $\partial H^2(\phi)$ of Eisenstein cohomology, associated to a suitable Hecke character $\phi$.

\subsection{Eisenstein cohomology}\label{sse:eis} Following the notations fixed in Sections~\ref{se:preli} and~\ref{se:nerve}, we define $Z:= Z({\rG})(\BR)^0$ and $\tilde{K}_{\infty}:=Z \cdot {K}_{\infty}$. Consider the irreducible representation $V_{\lambda}$ of $\rG_F$ of highest weight $\lambda$ and denote by $\omega$ its central character. We define the space 
\[
C^{\infty}(\G(\BQ) \backslash \G(\BA))(\omega^{-1})
\]
as the space of those $\BC$-valued, $C^{\infty}$-functions $f$ on $\G(\BQ) \backslash \G(\BA)$ which satisfy
$$f(zg)=\omega^{-1}(z)f(g)$$
for any $z \in Z(\rG)(\BA)^0$, $g \in \G(\BA)$. We define the archimedean component 
\[
C^{\infty}(\G(\BQ) \backslash \G(\BR))(\omega^{-1})
\]
and the non-archimedean component 
\[
C^{\infty}(\G(\BQ) \backslash \G(\BA_f))(\omega^{-1})
\]
of this space in an analogous way. 

We will also need to consider the subspace 
\[
\mathcal{A}(\G(\BQ) \backslash \G(\BR)) \subset C^{\infty}(\G(\BQ) \backslash \G(\BA))
\]
of \emph{automorphic forms} on $\rG$ (see e.g.~\cite[6.3]{GH24}).

If we denote by $\Fg$ the complexified Lie algebra of $\G(\BR)$, there is a canonical isomorphism 
\[
H^{\bullet}(S(\BC), V_{\lambda, \BC}) \simeq H^{\bullet}(\Fg, \tilde{K}_{\infty}, C^{\infty}(\G(\BQ) \backslash \G(\BA))(\omega^{-1}) \otimes V_{\lambda, \BC}) 
\]
(it is the adelic, infinite-level version of~\cite[VII, Cor.~2.7]{BoWa}). It induces an isomorphism of $\G(\BA_f)$-modules
\begin{equation}\label{gKiso}
\resizebox{0.91 \textwidth}{!}{
$H^{\bullet}(S(\BC), V_{\lambda, \BC}) \simeq C^{\infty}(\G(\BQ) \backslash \G(\BA_f))(\omega^{-1}) \otimes H^{\bullet}(\Fg, \tilde{K}_{\infty}, C^{\infty}(\G(\BQ) \backslash \G(\BR))(\omega^{-1}) \otimes V_{\lambda, \BC}) \,.$}
\end{equation}

This section aims to exploit the latter description to determine \emph{Eisenstein cohomology}, i.e. the image of the map of $\G(\BA_f)$-modules 
\begin{equation} \label{restrmap}
r:H^{\bullet}(\rS(\BC), V_{\lambda}) \rightarrow \partial H^{\bullet}(\rS(\BC),V_{\lambda})\,.
\end{equation}
To do so, we will adopt the following notation. For any algebraic Hecke character $\phi: \rT(\BQ) \backslash \rT(\BA) \rightarrow \BC^{\times}$ on $\rT$, we denote
\begin{equation} \label{boundmod}
I_{\phi}:=\Ind_{\rB(\BA_f)}^{\rG(\BA_f)} \overline{\BQ}_{\phi}
\end{equation}
Moreover, we define the \emph{Harish-Chandra module}

\resizebox{0.9\textwidth}{!}{
$I_{\phi,\infty}:=\left\{ f:\G(\BR) \rightarrow \BC \ \vert \ f \ \text{is} \  \tilde{K}_{\infty}-\text{finite and s.t.} 
 f(bg)=\phi_{\infty}(b)f(g) \ \forall b \in \rB(\BR), g \in \G(\BR) \right\}\,.$
}

We will then consider the space 
\[
I_{\phi}^*:=\Ind_{\rB(\BA)}^{\rG(\BA)} \phi = I_{\phi,\BC} \otimes I_{\phi,\infty}
\]

\subsection{The Eisenstein operator and its constant term}
Consider the element $\delta$ defined in~\eqref{halfsum}, which induces a map, still denoted by $\delta$
\[
T(\BA) \rightarrow \BI_F
\]
that actually lands in $\BI_{\BQ}$. So, we can define the Hecke character 
\[
\begin{array}{ccc}
\dlt : \rT(\BA) & \rightarrow & \BR^{\times}_{> 0} \\
t & \mapsto & \vert \delta(t) \vert_{\BI_{\BQ}},.
\end{array}
\]
Note that its restriction to $\rT^\prime$, seen as a Hecke character of $F$, coincides with the idelic norm on $\BI_F$.

We extend this to a map $\vert \cdot \vert : \rG(\BA) \rightarrow \BR^{\times}$, considering a suitable maximal compact subgroup $K_{\BA}$ of $\rG(\BA)$ and the associated Iwasawa decomposition
\[
\rG(\BA)=K_{\BA} \rT(\BA) \rm{U}(\BA)
\]
and letting $\vert \cdot \vert$ be trivial on $\rm{U}(\BA)$ and on $K_{\BA}$. 

Now fix an algebraic Hecke character $\phi$ of $\rT$, and an element $\Phi \in I^*_{\phi}$. Then it is known~(\cite{Lan76}) that for $s \in \BC$ with a large enough real part, the infinite sum\footnote{Note that we are not using unitary induction. This explains the absence of the half-sum of positive roots in our formulae.}
\[
\sum\limits_{\gamma \in \rB(\BQ) \backslash \rG(\BQ)} \Phi(\gamma g) \vert \gamma g \vert^s
\]
is absolutely convergent for all $g \in \rG(\BA)$. Letting $g$ vary, this allows one to define an element
\[
\Eis_{\phi,s}(\Phi) \in \mathcal{C}^{\infty}(\rG(\BQ) \backslash \rG(\BA))
\]
For any $s \in \BC$ with $\Re(s) >> 0$, one thus obtains the \emph{Eisenstein operator}
\begin{equation} \label{Eisop}
\Eis_{\phi,s} : I^*_{\phi \dlt^s} \rightarrow \mathcal{A}(\G(\BQ) \backslash \G(\BA))
\end{equation}
As a function of $s$, it admits a meromorphic continuation to the whole of $\BC$. 

Now recall the element $\theta$ of~\eqref{longestWeyl}. It is well known that the evaluation of the \emph{constant term} of the operator $\Eis_{\phi,s}$ at $\Psi \in I^*_{\phi \vert \delta \vert^s}$, formally defined by 
\[
g \mapsto \int\limits_{\rm{U}(\BQ) \backslash \rm{U}(\BA)} \Eis_{\phi,s}(\Psi)(ug) \mbox{d}u, \ \ \ g \in \rG(\BA)
\]
can be rewritten as
\[
g \mapsto \Psi(g)+\int\limits_{\rm{U}(\BA)} \Psi(\theta \cdot ug) \mbox{d}u, \ \ \ g \in \rG(\BA)
\]
and that $\Eis^*_{\phi,s}$ has a pole at $s=0$ if and only if this is the case for the constant term described above. To study whether such a pole occurs, we introduce the following.
\begin{defi} \label{intert}
\begin{enumerate}[wide, labelwidth=!, labelindent=0pt]\rm $\,$
Let $s \in \BC$ with $\Re(s) >> 0$.
\item The \emph{intertwining operator}
\[
T(\phi, s): I^*_{\phi \dlt^s} \rightarrow I^*_{\phi^{-1} \dlt^{-2-s}}
\]
is defined, for every $\Psi \in I^*_{\phi \dlt^s}$ and for every $g \in \G(\BA)$, by
\[
T(\phi, s)(\Psi(g)) = \int\limits_{\rm{U}(\BA)} \Psi(\theta \cdot ug) \mbox{d}u\,.
\]
\item 
For any place $v$ of $\BQ$, we define
\[
I^*_{\phi_v \dltv^s}:=\Ind_{\rB(\BQ_v)}^{\rG(\BQ_v)} \phi_v \dltv^s
\]
so that $I^*_{\phi \dlt^s}=\bigotimes\limits_v I^*_{\phi_v \dltv^s}$. Then, the local components $T(\phi_v,s)$ of the intertwining operator $T(\phi,s)$ are defined for all $\Psi_v \in I^*_{\phi_v \dltv^s}$ and for all $g_v \in \G(\BQ_v)$, by
\[
T(\phi_v,s)(\Psi_v(g_v)) = \int\limits_{\rm{U}(\BQ_v)} \Psi_v(\theta_v \cdot u_vg_v) \mbox{d}u_v\,
\]
and we have $T(\phi,s)=\bigotimes\limits_v T(\phi_v,s)$.
\end{enumerate}
\end{defi}

The appearance of a pole at $s=0$ in the intertwining operator is controlled by a quantity of arithmetic nature. To define it, we use the character $\omega_{F \vert \BQ}$ introduced in Subsection~\ref{basic}, and identify Hecke characters of $\rT^\prime$ with Hecke characters of $F$ according to the conventions of Subsection~\ref{algHecke}.

\begin{defi}$\,$\rm
\begin{enumerate}[wide, labelwidth=!, labelindent=0pt]
\item Let $\phi$ be the Hecke character of $\rT$, and denote in the same way the Hecke character of $F$ obtained by restricting $\phi$ to $\rT^\prime$. Denote by $\phi_{\BQ}$ the restriction of $\phi$ to a Hecke character of $\BQ$.

The meromorphic function $c(\phi,s)$ of the complex variable $s \in \BC$ is defined as
\[
c(\phi, s):= \frac{L(\phi,s-1) \cdot L(\phi_{\BQ} \cdot \omega_{F \vert \BQ},2s-2)}{L(\phi,s) \cdot L(\phi_{\BQ} \cdot \omega_{F \vert \BQ},2s-1)}\,.
\]

We denote by $\tilde{c}(\phi, s)$ the product of $c(\phi,s)$ with the appropriate $\Gamma$-factors of the occurring $L$-functions. 
\item For every place $v$ of $\BQ$, the local factor $c_v(\phi,s)$ of $c(\phi,s)$ is defined as
\[
c_v(\phi, s):= \frac{L_p(\phi,s-1) \cdot L_p(\phi_{\BQ} \cdot \omega_{F \vert \BQ},2s-2)}{L_p(\phi,s) \cdot L_p(\phi_{\BQ} \cdot \omega_{F \vert \BQ},2s-1)}
\] 
if $v$ corresponds to a finite prime $p$, and as the product of the appropriate $\Gamma$-factors of the occurring $L$-functions, if $v$ is an infinite place. Here, $L_p(\cdot, \cdot)$ denotes the corresponding factor in the Euler product of the $L$-function under consideration.
\end{enumerate}
\end{defi}

\begin{lemma} \label{intert_pole}
Let $\phi$ be a Hecke character of $\rT$. The intertwining operator $T(\phi, s)$ has a pole at $s=0$ if and only if this is the case for $c(\phi,s)$. 
\end{lemma}
\begin{proof}
Let $v=p$ be an odd prime that does not ramify in $F$ and at which $\phi$ is not ramified. Then consider the \emph{spherical function} $\Psi^{(0)}_v \in I^*_{\phi_v \dlt_v^s}$ defined for every $g_v=b_v k_v$ with $b_v \in \rB(\BQ_v)$ and $k_v \in \G(\BZ_v)$, by
$\Psi^{(0)}_v(g_v)=\phi_v \dlt_v^s(b_v)\,.$
The statement then follows from Lai's spelling out of \emph{Langlands' constant term formula} in our situation (\cite[Sect.~3]{Lai80}), saying that
\[
T(\phi_v,s)(\Psi^{(0)}_v)=c_v(\phi,s) \cdot \Psi^{(0)}_v.
\]
\end{proof}

Hence, we must study the zeros and poles of the $L$-functions appearing in $c(\phi,s)$. Since the Hecke characters of interest will be those contributing to the description of boundary cohomology of $\rS_K$, Theorem~\ref{thm:GH1} leads us to consider types of the form $w \star \lambda$, where $w$ is an element of the Weyl group of $\rG_F$ and the characters $w \star \lambda$ have been defined and computed in Subsection~\ref{KostantRep}. We will only need to consider the case of elements $w$ of length $\geq 2$.

By applying the basic facts about $L$-functions of Hecke characters recalled in Section~\ref{Heckechar}, we find the following. 

\begin{lemma}\label{condpole1}
\begin{enumerate}[wide, labelwidth=!, labelindent=0pt]
Let $\lambda=(k_1,k_2)$ be a character of $\T^\prime$ and $w$ be an element of the Weyl group $W$ of $\rG_{F}$ of length $\geq 2$. Fix an Hecke character $\phi$ of $F$ of type $w \star \lambda$. 
\item The half-plane of convergence for the Euler product of $L(\phi,s)$ is given by:
\begin{itemize}
\item $\mbox{\emph{Re}} \ s> -\frac{1}{2}-\frac{k_2}{2}$ if $w=$ \emph{(1\ 2\ 3)};
\item $\mbox{\emph{Re}} \ s> -\frac{1}{2}-\frac{k_1-k_2}{2}$ if $w=$ \emph{(1\ 3\ 2)};
\item $\mbox{\emph{Re}} \ s> -1-\frac{k_1}{2}$ if $w=$ \emph{(1\ 3)}.
\end{itemize}
\item The $L$-function $L(\phi,s)$ is entire, unless $(k_1,k_2)=(0,0)$, $w=$ \emph{(1\ 3)} and $\phi(\cdot)=\vert\cdot\vert_{\BI_F}^2$. In this case, it has a unique first-order pole at $s=-1$. 
\end{enumerate}
\end{lemma}

Moreover, taking $\phi$ as above, its restriction $\phi_{\BQ}$ to $\BI_{\BQ}$ has weight $2 w(\phi)$. The same then holds for $\phi_{\BQ} \cdot \omega_{F \vert \BQ}$. Hence, we obtain the following.  

\begin{lemma} \label{condpole2}
\begin{enumerate}[wide, labelwidth=!, labelindent=0pt]
Let $\lambda=(k_1,k_2)$ be a character of $\T^\prime$ and $w$ be an element of the Weyl group $W$ of $\rG_{F}$ of length $\geq 2$. Fix an Hecke character $\phi$ of $F$ of type $w \star \lambda$. 
\item The half-plane of convergence for the Euler product of $L(\phi_{\BQ} \cdot \omega_{F \vert \BQ},s)$ is given by:
\begin{itemize}
\item $\mbox{\emph{Re}} \ s> -2-k_2$ if $w=$ \emph{(1\ 2\ 3)};
\item $\mbox{\emph{Re}} \ s> -2-(k_1-k_2)$ if $w=$ \emph{(1\ 3\ 2)};
\item $\mbox{\emph{Re}} \ s> -3-k_1$ if $w=$ \emph{(1\ 3)}.
\end{itemize}
\item The $L$-function $L(\phi \cdot \omega_{F \vert \BQ},s)$ is entire, unless
\[ 
k_2=0, \ w= \mbox{\emph{(1\ 2\ 3)}} \ \mbox{and} \ \phi_{\BQ} \cdot \omega_{F \vert \BQ}( \cdot ) = \vert \cdot \vert_{\BI_\BQ}^3\,,
\]
or 
\[ 
k_1=k_2, \ w= \mbox{\emph{(1\ 3\ 2)}} \ \mbox{and} \  \phi_{\BQ} \cdot \omega_{F \vert \BQ}( \cdot ) = \vert \cdot \vert_{\BI_\BQ}^3\,.
\]
In each of the above two cases, it has a unique first-order pole at $s=-2$. 
\end{enumerate}
\end{lemma}

From the previous two lemmas, we immediately get the following corollary.
\begin{coro} \label{coro:nopole}
Let $k$ be a positive integer, put $\lambda=(k,0)$, and let $\phi$ be a Hecke character of $F$ of type $\mbox{\emph{(1\ 2\ 3\ )}} \star \lambda$. Suppose $\phi_{\BQ} \cdot \omega_{F \vert \BQ}( \cdot ) = \vert \cdot \vert_{\BI_\BQ}^3$. 

If $L(\phi,-1) =0$, then $c(\phi,s)$ has no pole at $s=0$.
\end{coro}

\subsection{Automorphic description of Eisenstein cohomology} To describe the main result of this section, we consider a modification of the intertwining operators $T(\phi,s)$ introduced in the previous subsection.
\begin{defi} \label{tloc}\rm
Let $\lambda$ be a character of $\rT^\prime$ and $w$ be an element of the Weyl group $W$ of $G_{F}$ of length $\geq 2$. Let $\phi$ be a Hecke character of $F$ of type $w \star \lambda$. Let $s \in \BC$ belong to the half-plane of convergence of the Euler product of $L(\phi,s)$. For every place $v$, we define 
\[
T^{\loc}(\phi_v,s):=c_v(\phi,s)^{-1} \cdot T(\phi_v,s)\,.
\]
We then put
\[
T^{\loc}(\phi,s):=\bigotimes\limits_v T^{\loc}(\phi_v,s)\,.
\]
\end{defi}
\begin{rmk}\label{tloc0}$\,$\rm
\begin{enumerate}
    \item The definition makes sense because if $s$ satisfies the hypotheses, i.e. if $\mbox{Re} \ s$ satisfies the bounds of Lemma~\ref{condpole1}(1), then both the operator $T(\phi_v,s)$ and the quantity $c_v(\phi,s)$ are indeed defined and non-zero (cfr.~\cite[Lem.~2.3.1]{Har87b}).
\item The operator $T^{\loc}(\phi,s)$ is also defined, and non-zero, for $s=0$.
\end{enumerate}
\end{rmk}

For the remainder of this section, we assume that $\lambda$, $w$, and $\phi$ are as in the preceding definition. 

\begin{defi} \label{Eis}\rm
The operator $\Eis_{\phi} : I_{\phi}^* \rightarrow \mathcal{A}(\G(\BQ) \backslash \G(\BA))$
is defined as 
\begin{enumerate}
\item \label{Eis1} the evaluation $\Eis_{\phi,0}$ at $s=0$ of the Eisenstein operator $\Eis_{\phi,s}$ of~\eqref{Eisop}, if $c(\phi,s)$ does not have a pole at $s=0$;
\item \label{Eis2} the residue at $s=0$ of the operator $s \cdot \Eis_{\phi,s}$, if $c(\phi,s)$ has a pole at $s=0$.
\end{enumerate}
\end{defi}

\begin{defi} \rm $\,$ 
\begin{enumerate}
\item We denote $I^\prime_{\phi}$ the submodule of $I_{\phi}$ generated by those tensors, whose component at $v$ belongs to the kernel of $T^{\loc}(\phi_v,0)$ for at least one finite place $v$. We denote $I^{\prime,*}_{\phi}:=I^\prime_{\phi} \otimes I_{\phi_{\infty}}$. 
\item For any finite place $v$, we denote $\tilde{J}_{\phi_v}:= \mbox{Im} \ T^{\loc}(\phi_v,0)$, and
\begin{equation}\label{J_philoc}
J_{\phi} := \bigotimes_v \tilde{J}_{\phi,v}
\end{equation}
\end{enumerate}
\end{defi}
\begin{rmk}\label{keys} \rm
As explained in~\cite[p. 583]{Har87b}, the results of~\cite{Key84} imply that if $\phi \in \Phi_p$, then the spaces $I^\prime_{\phi}$, resp. $J_{\phi}$ are non-trivial, proper subspaces of the domain, resp. of the codomain of $T^{\loc}(\phi,0)$.
\end{rmk}

\begin{defi} \label{eispoles} \rm
Denote by $\Phi_{\mbox{\tiny{np}}}$ the set of those $\phi$'s such that $c(\phi,s)$ does not have a pole at $s=0$, and by $\Phi_{\mbox{\tiny{p}}}$ the set of those $\phi$'s such that $c(\phi,s)$ has a pole at $s=0$.
\begin{enumerate}
\item For each $\phi \in \Phi_{\mbox{\tiny{np}}}$, we denote by $\Eis_{\phi}$ the operator induced in $(\mathfrak{g},\tilde{K}_{\infty})$-cohomology by the operator $\Eis_{\phi}$ of Definition~\ref{Eis}.(\ref{Eis1}).
\item \label{red} For each $\phi \in \Phi_{\mbox{\tiny{p}}}$, we denote by $\Eis^\prime_{\phi}$ the restriction to $I^{\prime,*}_{\phi}$ of the operator $\Eis_{\phi}$ induced in $(\mathfrak{g},\tilde{K}_{\infty})$-cohomology by the operator $\Eis_{\phi}$ of Definition~\ref{Eis}.(\ref{Eis2}). 
\end{enumerate}
\end{defi}

We are now ready to state the main result of Harder's paper~\cite{Har87b}.
\begin{thm}\cite[Thm.~2]{Har87b}\label{Eiscoh}
Let $r$ be the restriction morphism of~\eqref{restrmap}. 
Then, $r_{\BC} \circ ( \oplus_{\phi \in \Phi_{\mbox{\tiny{np}}}} \Eis_{\phi} \oplus_{\phi \in \Phi_{\mbox{\tiny{p}}}} \Eis^\prime_{\phi}) $ induces an isomorphism 
\begin{equation}
\Img(r) \simeq \bigoplus\limits_{w \in \CW} \bigoplus_{\substack{\mbox{\emph{\tiny{type}}}(\phi)= w \star \lambda \\
\ell(w) \geq 2 \\
\phi \in \Phi_{\mbox{\tiny{np}}}}} I_{\phi} \bigoplus_{\substack{\mbox{\emph{\tiny{type}}}(\phi)= w \star \lambda \\
\ell(w) \geq 2 \\
\phi \in \Phi_{\mbox{\tiny{p}}}}} I^\prime_{\phi} \oplus J_{\phi}
\end{equation}
\end{thm}

\begin{rmk}$\,$\rm
\begin{enumerate}
\item The case $\phi \in \Phi_{\mbox{\tiny{np}}}$ of the above theorem is proven in~\cite[pp. 581-583]{Har87b}, and the case $\phi \in \Phi_{\mbox{\tiny{p}}}$ in~\cite[pp. 583-584]{Har87b}.
\item \label{deg}
In the above theorem, the spaces $I_{\phi}$ and $I^\prime_{\phi}$ contribute to cohomology in degrees corresponding to the length of the corresponding Weyl element, hence degree 2 (in which case the space $J_{\phi}$ contributes to cohomology in degree 1) or 3 (in which case the space $J_{\phi}$ contributes to cohomology in degree 0).
\item The map in the statement of the above theorem induces an isomorphism between the extensions of scalars to $\BC$ of the spaces under consideration. One then obtains the desired isomorphism by passing to suitable $\overline{\BQ}$-structures.
\end{enumerate}
\end{rmk}

For $i \in \{0, \dots, 3 \}$, we will denote 
\begin{equation} \label{eiscoh}
H^i_{\Eis}(\rS(\BC), V_{\lambda}) := \mbox{Im} (H^i(\rS(\BC), V_{\lambda}) \rightarrow \partial H^i(\rS(\BC), V_{\lambda}))
\end{equation}
where the image is taken with respect to the restriction morphism. We call this space \emph{Eisenstein cohomology}, and its counterpart at finite level is obtained by taking $K$-invariants, 
\[
H^{\bullet}_{\Eis}(\rS_K(\BC), V_{\lambda})=H^{\bullet}_{\Eis}(\rS(\BC), V_{\lambda})^K
\]

With this notation, the first concrete consequence of Theorem~\ref{Eiscoh} that we will use is the following. 

\begin{coro} \label{boundphiHS_sub}
Let $k$ be a positive integer and let $\phi$ be a Hecke character of $F$ of infinity type $(k, -k-3)$. Suppose $\phi_{\BQ} \cdot \omega_{F \vert \BQ}( \cdot ) = \vert \cdot \vert_{\BI_\BQ}^3$. 

If $L(\phi,-1) =0$, then there exists a canonical monomorphism of $\rG(\BA_f)$-modules
\[
I_{\phi} \hookrightarrow H^2_{\Eis}(\rS(\BC), V_{k,\overline{\BQ}})
\]

For any compact open $K \subset \rG(\BA_f)$ which is small enough, it endows $I^K_{\phi}$ with a pure Hodge structure, isomorphic to a direct sum of 1-dimensional pure $\overline{\BQ}$-Hodge structures
\[
I_{k,-k-3} \otimes_F \overline{\BQ}
\]
of type $(1,k+2)$.
\end{coro}

\begin{proof}
Let $\phi$ be as in the hypotheses and observe that by putting $\lambda=(k,0)$, its type can be written as $\mbox{(1\ 2\ 3\ )} \star \lambda$. Then, by Corollary~\ref{coro:nopole}, $c(\phi,s)$ has no pole in $s=0$. Hence, by Theorem~\ref{Eiscoh}, $I_{\phi}$ appears as a submodule of degree 2 Eisenstein cohomology. Once one remembers that by Corollary~\ref{boundHS_Vk}, the pure 1-dimensional $F$-Hodge structure $I_{k,-k-3}$ has Hodge type $(1,k+2)$, the description of the Hodge structure on $I^K_{\phi}$ comes from the isomorphism~\eqref{HodgeInd}. 
\end{proof}

To state the second consequence of Theorem~\ref{Eiscoh} which will be of interest to us, denote, for $i \in \{0, \dots, 3 \}$
\begin{equation} \label{eiscoh_c}
H^i_{c,\Eis}(\rS(\BC), V_{\lambda}) := \mbox{Coker} (H^{i}(\rS(\BC), V_{\lambda}) \rightarrow \partial H^i(\rS(\BC), V_{\lambda}))
\end{equation}
where again the cokernel is taken with respect to the restriction morphism. Moreover, for $\phi$ a Hecke character of $F$ of type $\lambda$, define $\theta \phi$ by
\[
\theta \phi := \theta(\phi \vert \delta \vert) \vert \delta \vert^{-1}
\]

\begin{coro} \label{boundphiHS_sub_dual}
Let $k$ and $\phi$ be as in the Corollary above. Then there exists a canonical monomorphism of $\rG(\BA_f)$-modules
\[
I_{\theta \phi} \hookrightarrow H^1_{c,\Eis}(\rS(\BC), V_{k,\overline{\BQ}})
\]
exhibiting $I_{\theta \phi}$ as a direct summand of the space on the right. 

For any compact open $K \subset \rG(\BA_f)$ that is small enough, it endows $I^K_{\theta \phi}$ with a pure Hodge structure, a direct summand of the one on $H^1_{c,\Eis}(\rS_K(\BC), V_{k,\overline{\BQ}})$.  It is the direct sum of 1-dimensional pure $\overline{\BQ}$-Hodge structures
\[
I_{k+1,-k-2} \otimes_F \overline{\BQ}
\]
of type $(0,k+1)$
\end{coro}

\begin{proof}
Note that if $\phi$ is of infinity type $(k,-k-3)$, that is, of type $\mbox{(1\ 2\ 3\ )} \star (k,0)$, then $\theta \phi$ is of type $\mbox{(1\ 2\ )} \star (k,0)$, that is, of infinity type $(k+1,-k-2)$. Then, the statement is proven by applying Corollary~\ref{boundaut_Vk} and Theorem~\ref{Eiscoh}, reasoning analogously to the proof of Corollary~\ref{boundphiHS_sub}, and noticing that Corollary~\ref{boundaut_Vk} exhibits $I_{\theta \phi}$ as a direct summand of boundary cohomology.
\end{proof}

\begin{notation}
Fix $\phi$ as in the above corollary and $K$ small enough. We will denote the pure Hodge structure having $I^K_{\phi}$ as underlying vector space by
\begin{equation} \label{boundphiHS}
\partial H^2(\phi)
\end{equation}
and the one having $I^K_{\theta \phi}$ as underlying vector space by
\begin{equation} \label{boundphiHS_dual}
\partial H^1(\theta \phi)\,.
\end{equation}
\end{notation}

\begin{rmk} \label{conj_boundcoh} \rm
Suppose $\phi_{\BQ}$ satisfies the same assumptions as in Corollary~\ref{boundphiHS_sub}, but $\phi$ is of infinity type $(-k-3,k)$. Then, using the local system $V^{\vee}_k$ instead of $V_k$, and employing Corollary~\ref{boundHS_Vkdual}, the same proof shows the existence of a sub-Hodge structure $\partial H^2(\phi)$ of degree-2 boundary cohomology, of Hodge type $(k+1,0)$ and of underlying vector space $I^K_{\phi}$.
\end{rmk}


\section{A subspace of interior cohomology}\label{se:intcoh}
The aim of this section is to construct specific sub-Hodge structures $H^2(\phi)_!$ of interior cohomology of Picard surfaces. These subspaces are associated with Hecke characters $\phi$ of the imaginary quadratic field $F$ satisfying the assumptions of our Theorem~\ref{mainthm}. They indeed provide the last ingredient, which will be needed in the next section to prove the theorem. 

\subsection{The setup} We fix a positive integer $k$. We will consider Hecke characters $\phi$ of $F$ of infinity type $(k, -k-3)$. In particular, the weight of such a $\phi$ is $-3$ and the center of its functional equation is $s=-1$. 

For any such $\phi$, we identify it with a Hecke character on the maximal torus $\rT^\prime$ of $\rG^{\der}$ according to the conventions of Subsection~\ref{algHecke}. Moreover, we fix any extension of $\phi$ to a Hecke character of the maximal torus $\rT$ of $\rG$, and denote it by the same symbol; the context will make it clear when we need to see $\phi$ as a Hecke character of $\rT$ instead of just $F$. 

The first hypothesis on our $\phi$'s will be that the restriction $\phi_{\BQ}$ of $\phi$ to $\BQ$ satisfies
\begin{equation}\label{restrcond}
\phi_{\BQ}= \omega_{F \vert \BQ} \vert \cdot \vert^3_{\BI_{\BQ}}.
\end{equation}

\begin{rmk} \label{meaning_restrcond} \rm
Let us discuss the existence of Hecke characters verifying the above assumption. Observe that using the element $\delta$ defined in~\eqref{halfsum}, a Hecke character $\phi$ of $F$ of infinity type $(k,-k-3)$ can be written as 
\begin{equation} \label{unitphi}
\phi=\vert \delta \vert^{\frac{3}{2}} \cdot \phi_u
\end{equation}
where $\phi_u$ is a \emph{unitary} Hecke character. Then we have an equality of $\rG(\BA_f)$-representations 
\begin{equation}\label{indun}
\Ind_{\rB(\BA_f)}^{\rG(\BA_f)} \phi = \IndUn_{\rB(\BA_f)}^{\rG(\BA_f)} \phi_u \cdot \vert \delta \vert^{\frac{1}{2}}
\end{equation}
where $\IndUn$ denotes the \emph{unitary} induction functor.
Moreover, we have 
\begin{align*}
\phi_{u,\infty}(z)=z^{-k-\frac{3}{2}} \overline{z}^{k+\frac{3}{2}} \quad \text{for}\quad  z \in \BC^{\times}.
\end{align*} 

Now, define $\phi^{\perp}$ and $\phi_u^{\perp}$ as in~\eqref{phi_perp}. Then, the proof of~\cite[Lem.~6.9.2]{BC09} shows that: 
\begin{enumerate}
    \item[(a)] condition~\eqref{restrcond} is equivalent to asking that $\phi_u = \phi^{\perp}_u$, hence to asking 
\begin{equation}\label{autoadjoint_3}
\phi^{\perp}=\phi \vert \cdot \vert^{-3}_{\BI_F}, 
\end{equation} 
\item[(b)] such $\phi_u$'s verifying $\phi^{\perp}_u=\phi_u$ do exist.
\end{enumerate}
\end{rmk}

The second condition that we want to enforce on our $\phi$ is that 
\begin{equation}\label{vanishcond}
L(\phi,-1) = 0.
\end{equation}
Lemma~\ref{intert_pole} tells us that under~\eqref{restrcond}, the intertwining operator $T(\phi,s)$ of Definition~\ref{intert} may have a pole at $s=0$, but the same Lemma also implies that when~\eqref{vanishcond} holds, no pole can appear. Hence, under this hypothesis, $T(\phi,0)$ is well defined. 

\begin{defi} \rm
Let $k$ be a positive integer and $\phi$ be a Hecke character of $F$ of infinity type $(k,-k-3)$. If $\phi$ satisfies~\eqref{restrcond} and~\eqref{vanishcond}, then we denote 
\[
J^*_{\phi}:= \mbox{Im}(T(\phi,0))
\]
We denote $J_{\phi,v}$ its component at a finite place $v$, and $J_{\phi,f}$ the restricted product of the $J_{\phi,v}$'s over $v$.
\end{defi}
\begin{rmk}\label{Langlands_quotient}$\,$\rm
\begin{enumerate}
\item 
Let $p$ be a prime that is inert or ramified in $F$. Then $J_{\phi,p}$ is the \emph{Langlands quotient} of $I_{\phi,p}$, that is, its unique irreducible quotient as a $\rG(\BQ_p)$-module. By the results of~\cite{Key84} (as explained in~\cite[Sect.~12.2]{Rog90}), the module $J_{\phi,p}$ is a \emph{proper} quotient of $I_{\phi,p}$ precisely under the hypothesis~\eqref{restrcond}. 
\item
Recall the space $J_{\phi}$ introduced in~\eqref{J_philoc}. By construction, we have $J_{\phi,\BC} \simeq J_{\phi,f}$.
\end{enumerate}
\end{rmk}

Now suppose that $\phi$ satisfies~\eqref{restrcond}, or equivalently~\eqref{autoadjoint_3}. By Remark~\ref{selfadjoint}, if the sign $\epsilon(\phi)$ satisfies
\begin{equation} \label{signcond}
\epsilon(\phi)=-1
\end{equation}
then~\eqref{vanishcond} holds. Now, consider the space $\rS$ introduced in Subsection~\ref{ss:autbound}. Recall that for any local system $V_{\lambda}$, the \emph{cuspidal cohomology} 
\[
H^{\bullet}_{\cusp}(S(\BC), V_{\lambda, \BC})
\]
is the $\rG(\BA_f)$-submodule of interior cohomology $H_!^{\bullet}(S(\BC),  V_{\lambda, \BC})$ (cfr.~\eqref{intcoh}), defined as the direct sum of all those irreducible $\rG(\BA_f)$-submodules of $H^{\bullet}_!(S(\BC), V_{\lambda, \BC})$ which are isomorphic to the non-archimedean component of some cuspidal automorphic representation of $\rG(\BA)$.  If we denote the set of these isomorphism classes by $\mathcal{C}_{\cusp}$, then by multiplicity one for the discrete spectrum of $\rG(\BA)$ (\cite[Thm.~13.3.1]{Rog90}), cuspidal cohomology decomposes as 
\begin{equation} \label{cuspcoh}
H^{\bullet}_{\cusp}(S(\BC), V_{\lambda,\BC}) \simeq \bigoplus\limits_{\substack{\pi=\pi_f \otimes \pi_{\infty} \\
\mbox{\tiny{s. t.}} \ \pi_f \in \mathcal{C}_{\cusp}}} \pi_f \otimes H^{\bullet}(\Fg, \K, \pi_{\infty} \otimes V_{\lambda,\BC}).
\end{equation}
In~\cite[p. 99]{Har93}, the following observation is made. As a consequence of~\cite[Thms.~13.3.6,~13.3.7]{Rog90}, corrected in~\cite[Thm.~1.1]{Rog92b}, it follows that under assumption~\eqref{signcond}, there exists a $\rG(\BA_f)$-submodule
\begin{equation} \label{H2(pi(phi)f)}
H^2(\pi(\phi)_f) \quad \text{of} \quad H^{2}_{\cusp}(S(\BC), V_k),
\end{equation}
 whose local components are isomorphic to $J_{\phi,v}$ for infinitely many places $v$. We explain this in the following subsection. 

\subsection{An application of Rogawski's results} 

Given the Hecke character $\phi$, we are now going to construct the $\rG(\BA_f)$-submodule $H^2(\pi(\phi)_f)$ mentioned in equation~\eqref{H2(pi(phi)f)}. This will amount to constructing a representation of $\rG(\BA)$ which we will show to be cuspidal and cohomological by applying Rogawski's results. 

For this, we fix a choice of a Hecke character $\mu$ of $F$ whose restriction to $\BQ$ is $\omega_{F \vert \BQ}( \cdot )$. Moreover, we choose a couple of characters $(\rho,\rho^\prime)$ on the norm-one subgroup $C_F$ of the id\`ele class group of $F$ as follows: define $\rho_1$ to be the restriction of $\mu \cdot \phi_{u}$ to $C_F$, and $\rho^\prime$ to be trivial. To such a couple, one attaches a character $\rho$ on $\mathrm{U}_2 \times \mathrm{U}_1$ as in~\cite[p. 395]{Rog92b}. 

\begin{defi} \rm
Let $\phi$ be a Hecke character of $F$ satisfying 
\[
\phi_{\BQ}= \omega_{F \vert \BQ} \vert \cdot \vert^3_{\BI_{\BQ}}
\]
For $v$ a finite place of $\BQ$, define the $\rG(\BQ_v)$-representation $\pi^n(\rho_v)$ in the following way:
\begin{enumerate}
    \item is $v$ is inert or ramified in $F$, then we define $\pi^n(\rho_v)$ as the Langlands quotient $J_{\phi,v}$ of $I_{\phi,v}$, (see Remark~\ref{Langlands_quotient});
    \item if $v$ splits into places $w_1$ and $w_2$ in $F$, then identifying $\rG(\BQ_v) \simeq \GL_3(F_{w_1}) \times F^{\times}_{w_1}$, and seeing $\phi_v$ as a character of $(F^{\times}_{w_1})^4$, we define $\pi^n(\rho_v)$ as the representation induced from the standard $(2,1)$ parabolic of $\GL_3(F_{w_1})$ by the character 
\[
\left( \left(
\begin{array}{cc}
M & * \\
 & \lambda
\end{array}
\right), f\right) \mapsto \phi_v(\det(M),1,\lambda,f)
\]
\end{enumerate}
\end{defi}
\begin{rmk} \label{rog_finite} \rm
For any finite place $v$, the unitarily induced representation called $\ind(\eta_v)$ in~\cite[p. 395]{Rog92b} is nothing but our $I_{\phi_v}$, or more precisely, its restriction to a representation of $\rG^{\nu=1}$. This follows by observing that in the notation of~\eqref{indun},  the character called $\phi$ in \emph{loc.~cit.} becomes the character $\phi_u$ defined in~\eqref{unitphi}. This shows that for $v$ inert or ramified, our notation $\pi^n(\rho_v)$ is coherent with Rogawski's one in~\cite[p.~395]{Rog92b}. The same holds, by definition, in the case of split $v$~(\cite[p.~396]{Rog92b}).
\end{rmk}

\begin{defi} \rm
The discrete series representation $\pi(\phi)_{\infty}$ of $\rG(\BR)$ is defined as the representation $\pi^s(\rho_{\infty})$ denoted by $\pi^s(\rho)$ in~\cite[p.~77, first line]{Rog92a}. 
\end{defi}

\begin{rmk} \label{rog_infinite} \rm
Let $\lambda=(k_1,k_2,c,r)$ be a character of $\rT_F$. By~\cite[3.2]{Rog92a}, the representation $\pi^s(\rho)$ is $V_{\lambda}$-cohomological in degree 2 when either $k_2=0$ or $k_1=k_2$. In the first case, it coincides with the discrete series representation denoted by $\pi^+$ in \emph{loc. cit.}; in the second case, it coincides with the one denoted by $\pi^-$. 

To see how this follows from \emph{loc. cit.}, observe that in our notation, the restriction to $\rG^{\nu=1}$ of the representation $V_{\lambda}$ is identified with a triple $(k_1,k_2,c)$, whereas in~\cite[p. 79]{Rog92a} it is identified with a triple $(m,r,n)$. Now, under the change of variables between the two parametrizations, one sees that $k_2=0$ if and only if $r-n=1$, and that $k_1=k_2$ if and only if $m-r=1$.    
\end{rmk}

For the following statement, recall the level $K$ Hecke algebra defined in Remark~\ref{Hecke_algebra}. Moreover, recall that if $\phi$ satisfies~\eqref{restrcond}, or equivalently~\eqref{autoadjoint_3}, then by Remark~\ref{selfadjoint} there is a well-defined sign $\epsilon(\phi)$. 
\begin{prop} \label{rogmod}
Let $k$ be a positive integer and $\phi$ a Hecke character of $F$ of type $(k,-k-3)$, satisfying
\[
\phi_{\BQ}= \omega_{F \vert \BQ} \vert \cdot \vert^3_{\BI_{\BQ}}
\]
and 
\[
\epsilon(\phi)=-1
\]

Let the $\rG(\BA_f)$-representation $\pi_f(\phi)$ and the $\rG(\BA)$-representation $\pi(\phi)$ be defined by 
\begin{align}\label{pi}
\pi(\phi)_f& :=\bigotimes\limits_{v} \pi^n(\rho_v) \,,\nonumber\\
\pi(\phi)& :=\pi(\phi)_f \otimes \pi(\phi)_{\infty}
\end{align}
Then  
\begin{enumerate}[wide, labelwidth=!, labelindent=0pt]
\item The representation $\pi(\phi)$ is automorphic and has multiplicity 1 in the cuspidal spectrum of $\rG$.
\item \label{cohmod} There exist a $\overline{\BQ}$-model of $\pi_f(\phi)$, denoted by the same symbol, a 1-dimensional $\overline{\BQ}$-Hodge structure $H^2(\phi)_!$ of type $(k+2,0)$, and, for any neat open compact $K$ small enough, a submodule under the Hecke algebra of level $K$
\[
H^2(\pi(\phi)_f) \hookrightarrow H^{2}_{!}(\rS_K(\BC), V_k \otimes_F \bar{\BQ})
\]
such that
\begin{equation} \label{innermod}
H^2(\pi(\phi)_f) \simeq \pi(\phi)_f^K \otimes H^2(\phi)_!
\end{equation}
as a Hecke module and as a Hodge structure. 
\end{enumerate}
\end{prop}
\begin{proof}
\begin{enumerate}[wide, labelwidth=!, labelindent=0pt]
\item The automorphic representation $\pi(\phi)$ belongs to the $A$-packet $\Pi(\rho)$ of~\cite[p.~396]{Rog92b}, and by construction, the number $n(\pi)$ of \emph{loc. cit.} of places $v$ where $\pi(\phi)_v$ is isomorphic to $\pi^s(\rho_v)$ is equal to 1 (since this happens only for $v=\infty$). Now write $\phi_{R}$ for the character denoted $\phi$ in \emph{loc. cit.}. By Remark~\ref{rog_finite}, we have $\phi_R=\phi^{(1)}_u$, so that by~\eqref{unitphi}
\[
L\left(\phi_R,\frac{1}{2}\right)=L\left(\phi \cdot \vert \cdot \vert_{\BI_F}^{-\frac{3}{2}},\frac{1}{2}\right)=L\left(\phi,-1\right)
\]
which shows that our hypothesis $\epsilon(\phi)=-1$ is precisely equivalent to $\epsilon(\frac{1}{2},\phi_R)=-1$. By~\cite[Thm. 1.1]{Rog92b}, we obtain that the multiplicity of $\pi(\phi)$ in the discrete spectrum of $\rG$ is 1. By~\cite[Thm. 13.3.6 (a)]{Rog90}, it is actually a representation appearing in the \emph{cuspidal} spectrum.
\item Consider the decomposition~\eqref{cuspcoh} of $V_k$-valued cuspidal cohomology of $\rS$. Upon choosing a $\overline{\BQ}$-model for each submodule $\pi_f$, denoted by the same symbol, we get a $\overline{\BQ}$-model for cuspidal cohomology, described as 
\begin{equation}
H^{\bullet}_{\cusp}(\rS(\BC), V_{k, \overline{\BQ}}) := \bigoplus\limits_{\substack{\pi=\pi_f \otimes \pi_{\infty} \\
\mbox{\tiny{s. t.}} \ \pi_f \in \mathcal{C}_{\cusp}}} \pi_f \otimes \Hom_{\rG(\BA_f)} (\pi_f, H^{\bullet}_{!}(\rS(\BC), V_{k, \overline{\BQ}})) 
\end{equation}
where each space $\Hom_{\rG(\BA_f)} (\pi_f, H^{\bullet}_{!}(\rS(\BC), V_{k, \overline{\BQ}}))$ is endowed with a canonical $\overline{\BQ}$-Hodge structure. 
Now consider the submodule corresponding to $\pi=\pi(\phi)$. By Remark~\ref{rog_infinite}, $\pi(\phi)_{\infty}$ is $V_k$-cohomological in degree 2. By the comparison isomorphism between singular interior cohomology and \'etale interior cohomology (see e.g.~\cite[Sect.~1.7]{BR92}, the dimension of the Hodge structure
\begin{equation} \label{intphiHS}
H^2(\phi)_!:=\Hom_{\rG(\BA_f)} (\pi(\phi)_f, H^{2}_{!}(\rS(\BC), V_{k, \overline{\BQ}}))
\end{equation}
is the same as the dimension of the space $V^2(\pi_f))$ of Case 5 of~\cite[pp.~91-92]{Rog92a}, which according to \emph{loc. cit.} is equal to 1, by our choice of the archimedean component of $\pi(\phi)$. 
We then obtain the desired Hecke submodule. Finally, to get the assertion on the type of $H^2(\phi)_!$, we apply the following two facts: 

\emph{(a)} if $r$ is the weight of the pure Hodge structure on $V_{\lambda}$ (so that, by~\eqref{weightrep}, $r=k$ in our case), then $H^2_!(\rS(\BC), V_{\lambda})$ has weight $r+2$. Thus, its Hodge substructure $H^2(\pi(\phi)_f)$ has the same weight; 

\emph{(b)} the Hodge types of $H^2(\pi(\phi)_f)$ are those $(p,q)$'s appearing in the $(p,q)$-decomposition of the space $H^2(\Fg, \K, \pi(\phi)_{\infty} \otimes V_{k,\BC})$, which is in turn determined by the $(p,q)$-decomposition of the discrete series representation $\pi(\phi)_{\infty}$. Now, by Remark~\ref{rog_infinite} the discrete series representation $\pi(\phi)_{\infty}$ coincides with $\pi^+$, which is holomorphic by~(\cite[Sect.~1.3]{BR92}). 
\end{enumerate}
\end{proof}

\begin{rmk} \label{intphiHS_quot} \rm
Suppose the hypotheses of Proposition~\ref{rogmod} hold. For any small enough level $K$, the simple, level $K$ Hecke module $\pi^K(\phi)_f$ corresponds to a minimal ideal in the level $K$ Hecke algebra. Take an idempotent generator $e_{\pi(\phi)}$ of this ideal. By a standard argument (cfr. for example~\cite[Prop.~5.4]{Wil15}) we see that the Hodge structure $H^2(\phi)_!$ is isomorphic to the image of $H_!^2(\rS_K(\BC), V_k)$ under $e_{\pi(\phi)}$. Thus, we will be able to consider $H^2(\phi)_!$ both as a subobject and as a quotient of the Hodge structure on $H_!^2(\rS_K(\BC), V_k)$. 
\end{rmk}

\begin{rmk} \label{conj_intcoh} \rm
Suppose $\phi_{\BQ}$ satisfies the assumption of Proposition~\ref{rogmod}, but with $\phi$ being of infinity type $(-k-3,k)$ instead of $(k,-k-3)$. The same proof, but employing the local system $V^{\vee}_k$ instead of $V_k$, shows the existence of a Hodge structure $H^2(\phi)_!$ of type $(0,k+2)$ satisfying all the other requirements in the statement. This follows from the computation of the highest weight of $V^{\vee}_k$ of Remark~\ref{conj_highestweight} and from the fact that using the latter local system, by Remark~\ref{rog_infinite}, the component at infinity $\pi(\phi)_{\infty}$ is isomorphic to the discrete representation $\pi^-$, which is \emph{anti}-holomorphic by~(\cite[1.3]{BR92}).
\end{rmk}

\section{Extensions associated to Hecke characters}\label{se:Heckext}   

In this last section, we continue considering Hecke characters $\phi$ of $F$ satisfying the hypotheses of the previous section, which will be recalled below. In particular, their $L$-function is supposed to vanish at the central point, and as a consequence, as discussed in~\ref{BBHecke_expl}, the Bloch-Beilinson conjectures predict the existence of a non-trivial extension of Hodge structures $E_{\phi}$ of geometric origin. 

Our first aim is to put together all the ingredients introduced so far, to construct candidates $E_{\phi}(\Psi)$ for $E_{\phi}$, by means of the cohomology of a local system on a Picard modular surface $\rS_K$. They depend on the choice of a cohomology class $\Psi$ in a suitable subspace of Eisenstein cohomology of $\rS_K$. This will prove Theorem~\ref{mainthm} stated in the introduction.

\subsection{Construction of extensions}

For a positive integer $k$ and a neat level $K \subset \rG(\BA_f)$, we look at the local system $V_k$ over the Picard modular surface $\rS_K$. Consider its interior cohomology, defined in~\eqref{intcoh^K}, and its Eisenstein cohomology, defined in~\eqref{eiscoh}. The extension of scalars to $\overline{\BQ}$ of the long exact sequence~\eqref{longexloc}
gives rise to a short exact sequence 
\begin{equation} \label{boundexseq}
 0 \rightarrow H_!^{2}(\rS_K(\BC), V_{k,\overline{\BQ}}) \rightarrow H^{2}(\rS_K(\BC), V_{k,\overline{\BQ}}) \rightarrow H_{\Eis}^{2}(\rS_K(\BC), V_{k,\overline{\BQ}}) \rightarrow 0
\end{equation}

To state our main theorem, recall that to a Hecke character $\phi$ satisfying the hypotheses formulated below, we have attached pure Hodge structures $\partial H^2(\phi)$ and $H^2(\phi)_!$, mentioned respectively in~\eqref{boundphiHS} and~\eqref{intphiHS}.

\begin{thm}\label{sourcext1}
Let $k$ be a positive integer and $\phi$ a Hecke character of $F$ of type $(k,-k-3)$, satisfying
\[
\phi_{\BQ}= \omega_{F \vert \BQ} \vert \cdot \vert^3_{\BI_{\BQ}}
\]
and 
\[
\epsilon(\phi)=-1
\]
\begin{enumerate}[wide, labelwidth=!, labelindent=0pt]
\item For any level $K$ small enough, the exact sequence~\eqref{boundexseq}
has a subquotient short exact sequence of Hodge structures 
\begin{equation} \label{boundseq}
0 \rightarrow H^2(\phi)_! \rightarrow \tilde{E}_0 \rightarrow \partial H^2(\phi) \rightarrow 0
\end{equation}

\item Let $m(\phi,K)$ be the dimension of $I^K_{\phi}$. Then the short exact sequence~\eqref{boundseq} induces an extension of $\overline{\BQ}$-Hodge structures
\begin{equation}\label{sourcext}
0 \rightarrow H_{\phi}(-1) \rightarrow E_0 \rightarrow \one^{\oplus m(\phi,K)} \rightarrow 0
\end{equation}
where $H_{\phi}$ is the 1-dimensional $\overline{\BQ}$-Hodge structure of type $(k,-k-3)$ attached to $\phi$. 
\item \label{ext_mainthm} Any element $\Psi$ of $I_{\phi}^K$ provides an extension $E_{\phi}(\Psi)$ of geometric origin of $\one$ by $H_{\phi}(-1)$ in the category $\MHS_{\overline{\BQ}}$. 
\end{enumerate}
\end{thm}
\begin{proof}
\begin{enumerate}[wide, labelwidth=!, labelindent=0pt]
\item We know that the hypothesis on $\phi_{\BQ}$ is equivalent to asking that $\phi$ verifies~\eqref{autoadjoint_3}. By Remark~\ref{selfadjoint}, there is a well-defined sign $\epsilon(\phi)$, and $\epsilon(\phi)=-1$ implies $L(\phi,-1)=0$. Thus, by Corollary~\ref{boundphiHS_sub}, the pure Hodge structure $\partial H^2(\phi)$ is a sub-Hodge structure of $H^2_{\Eis}(\rS_K(\BC), V_k)$. By Proposition~\ref{rogmod} and Remark~\ref{intphiHS_quot}, the pure Hodge structure $H^2(\phi)_!$ is a quotient Hodge structure of $H^2_!(\rS_K(\BC), V_k)$. Hence, the desired extension is obtained from the short exact sequence~\eqref{boundexseq} by taking pullback via the inclusion $\partial H^2(\phi) \hookrightarrow H^2_{\Eis}(\rS_K(\BC), V_k)$ and pushout via the projection $H^2_!(\rS_K(\BC), V_k) \twoheadrightarrow H^2(\phi)_!$. 
\item By Corollary~\ref{boundphiHS_sub}, the pure Hodge structure $\partial H^2(\phi)$ is a direct sum of copies of the 1-dimensional $\overline{\BQ}$-Hodge structure $I_{k,-k-3,\overline{\BQ}}$. By tensoring the extension of the previous point by the dual of $I_{k,-k-3,\overline{\BQ}}$, we get an extension 
\[
0 \rightarrow H^2(\phi)_! \otimes I^{\vee}_{k,-k-3,\overline{\BQ}} \rightarrow E_0 \rightarrow \one^{\oplus m(\phi,K)} \rightarrow 0
\]
Since the type of $H^2(\phi)_!$ is $(k+2,0)$ and the type of $I^{\vee}_{k,-k-3,\overline{\BQ}}$ is $(1,k+2)$, the type of the 1-dimensional $\overline{\BQ}$-Hodge structure $H^2(\phi)_! \otimes I^{\vee}_{k,-k-3,\overline{\BQ}}$ is 
$(k+1,-k-2)$. This tells us that as a $\overline{\BQ}$-Hodge structure, $H^2(\phi)_! \otimes I^{\vee}_{k,-k-3,\overline{\BQ}}$ is indeed isomorphic to $H_{\phi}(-1)$.
\item Choosing an element $\Psi$ of $I_{\phi}^K$ amounts to defining a morphism of Hodge structures $\one \rightarrow \one^{m(\phi,K)}$. By pullback of the extension~\eqref{sourcext} along this morphism, we get the desired extension $E_{\phi}(\Psi)$. To see that it is of geometric origin, denote by $M(V_k)$ the motive $e_{\lambda} M(A^N)$ obtained as in~\eqref{motive_lambda} when choosing the representation $V_{\lambda}$ to be equal to $V_k$. It has Hodge realization $R_H(M(V_k))=H^{\bullet}(\rS_K,V_k)$. Then, by construction, $E_{\phi}(\Psi)$ is a subquotient Hodge structure of the mixed Hodge structure $R_H(M(V_k))$, hence it is of geometric origin.
\end{enumerate}
\end{proof}

\begin{coro} \label{coroext}
Let $k$ and $\phi$ be as in Theorem~\ref{sourcext1}. Associating the extension $E_{\phi}(\Psi)$ to $\Psi \in I_{\phi}^K$ defines a morphism
\begin{equation} \label{extmor}
I^K_{\phi} \rightarrow \Ext^1_{\MHS_{\overline{\BQ}}}(\one, H_{\phi}(-1))
\end{equation}
\end{coro}

Using Remarks~\ref{conj_boundcoh} and~\ref{conj_intcoh}, one can run the same proof as the one of Theorem~\ref{sourcext1}, but employing $V^{\vee}_k$ instead of $V_k$, and deduce the following variant of point~(3) of the above Theorem.

\begin{coro} \label{conj_coroext}
Let $k$ and $\phi_{\BQ}$ be as in Theorem~\ref{sourcext1}, but suppose that $\phi$ is of infinity type $(-k-3,k)$. The choice of an element $\Psi$ of $I_{\phi}^K$ provides an extension $E_{\phi}(\Psi)$ of $\one$ by $H_{\phi}(-1)$ in the category $\MHS_{\overline{\BQ}}$. This defines a morphism
\begin{equation} \label{extmor_conj}
I^K_{\phi} \rightarrow \Ext^1_{\MHS_{\overline{\BQ}}}(\one, H_{\phi}(-1))
\end{equation}
such that all extensions in its image are of geometric origin.
\end{coro}

For the needs of the next subsection, we also need to construct a \emph{dual} extension in the following way. The extension of scalars to $\overline{\BQ}$ of the long exact sequence~\eqref{longexloc} gives rise to a short exact sequence 
\begin{equation} \label{boundexseq_d}
\partial H^{1}(\rS_K(\BC), V_{k,\overline{\BQ}}) \rightarrow H_c^{2}(\rS_K(\BC),V_{k,\overline{\BQ}}) \rightarrow H^{2}(\rS_K(\BC), V_{k,\overline{\BQ}})  
\end{equation}
Recall the pure Hodge structure $\partial H^1(\theta \phi)$ defined in~\eqref{boundphiHS_dual}.
It is a direct summand of $H^1_{\Eis,c}(\rS_K(\BC), V_{k,\overline{\BQ}})$, hence a quotient Hodge structure of the kernel of the second arrow in the above exact sequence.  A reasoning analogous to the proof of Theorem~\ref{sourcext1}, but replacing the use of Corollary~\ref{boundphiHS_sub} by the use of Corollary~\ref{boundphiHS_sub_dual}, and pullback by pushout via the epimorphism onto $\partial H^1(\theta \phi)$, shows the following. 

\begin{prop}\label{sourcext2}
Let $k$ and $\phi$ be as in Theorem~\ref{sourcext1}.
\begin{enumerate}[wide, labelwidth=!, labelindent=0pt]
\item For any level $K$ small enough, the short exact sequence~\eqref{boundexseq_d} has a subquotient short exact sequence of Hodge structures 
\begin{equation} \label{boundseq_d}
0 \rightarrow \partial H^1(\theta \phi) \rightarrow \tilde{E}_0^\prime \rightarrow  H^2(\phi)_! \rightarrow 0
\end{equation}

\item Let $m^\prime(\phi,K)$ be the dimension of $I^K_{\theta \phi}$. Then the short exact sequence~\eqref{boundseq_d} induces an extension of $\overline{\BQ}$-Hodge structures
\begin{equation}\label{sourcext_dual}
0 \rightarrow \one(1)^{\oplus m^\prime(\phi,K)} \rightarrow E_0^\prime \rightarrow H_{\phi}(-1) \rightarrow 0
\end{equation}
where $H_{\phi}$ is the 1-dimensional $\overline{\BQ}$-Hodge structure of type $(k,-(k+3))$ attached to $\phi$. 
\end{enumerate}
\end{prop}

\begin{rmk} \rm
Poincaré-Verdier duality provides a canonical isomorphism between boundary cohomology in degree 2 and the dual of boundary cohomology in degree 1. The above extension could have been obtained by applying such a duality to the extension of Theorem~\ref{sourcext1}.    
\end{rmk}

\begin{coro} \label{coroext_dual}
The choice of an element $\Psi^\prime$ of $I_{\theta \phi}^K$ provides, by pushout of the extension~\eqref{sourcext_dual}, an extension $E_{\phi}(\Psi^\prime)$ of $H_{\phi}(-1)$ by $\one(1)$ in the category $\MHS_{\overline{\BQ}}$. This defines a morphism
\begin{equation} \label{extmor_dual}
I^K_{\theta \phi} \rightarrow \Ext^1_{\MHS_{\overline{\BQ}}}(H_{\phi}(-1),\one)
\end{equation}
such that all extensions in its image are of geometric origin. 
\end{coro}

\section{Biextensions associated to Hecke characters}\label{biext}
Consider the extensions of geometric origin $E_{\phi}(\Psi)$ constructed in the previous section. To verify the Bloch-Beilinson conjectures for our Hecke characters $\phi$, one has to show that there exists an element $\Psi$ of $I^K_{\phi}$ such that the extension $E_{\phi}(\Psi)$ is non-trivial. This is the subject of work in progress by the authors, where, as discussed in the introduction, verifying the desired non-triviality will rest crucially on the non-vanishing of a certain \emph{biextension height} (\cite{Hai90}, \cite{BdJS23}) associated to $E_{\phi}(\Psi)$. 

The aim of this section is then to define the input for such a height, that is, a \emph{biextension} in which $E_{\phi}(\Psi)$ sits. First, we recall what this means in our setting (cfr. for example~\cite{BdJS23}, Def. 2.1).  
\begin{defi}
Let $H$ be a pure $\overline{\BQ}$-Hodge structure of weight $-1$. A \emph{biextension} of $H$ is a $\overline{\BQ}$-mixed Hodge structure $B$ together with an identification of its associated weight-graded semisimple Hodge structure
\[
\Gr^W B \simeq \one (1) \oplus H \oplus \one\,.
\]
\end{defi}

Now, we continue to look at the local system $V_k$ on the Picard modular surface $\rS_K$. As in Subsection~\ref{boundhodge}, we consider the Baily-Borel compactification $\rS^*_K$ together with the open immersion $j$ of $\rS_K$ and the closed immersion $i$ of the boundary $\partial \rS_K$. To simplify notation, we will denote the latter by $\partial$. In the following, we make free use of the six functor formalism for mixed Hodge modules mentioned in Subsection~\ref{subsec_hodge}. We will work in the derived category of mixed Hodge modules, with its canonical triangulated structure; all functors will be implicitly derived (e.g., we denote $Rj_*$ simply by $j_*$), and we will not use different symbols for cohomology and hypercohomology. 

We fix a decomposition of $\partial$ as the disjoint union of two sets of connected components 
\begin{equation} \label{sepbound}
\partial = \Theta \bigsqcup \Sigma\,.
\end{equation}

For any $\triangle \in \{\Theta, \Sigma \}$,  consider the pair of complementary, open respectively closed immersions 
\[
j^{\triangle} : \rS_K(\BC) \hookrightarrow \rS_K(\BC) \bigsqcup \ \triangle \hookleftarrow \ \triangle : i^{\triangle}
\]
and the open immersion
\[
j_{\triangle} : \rS_K(\BC) \bigsqcup \ \triangle \hookrightarrow \rS^*_K(\BC) 
\]
and define $i_{\triangle}:= j_{\triangle} \circ i^{\triangle}$. There are canonical exact triangles 
\[
j_{\triangle,*} j^{\triangle}_! \rightarrow j_* \rightarrow i_{\triangle,*} i^{\triangle,*} j^{\triangle}_* \rightarrow \,
\]

For any local system $V_{\lambda}$ on $\rS_K$, denote 
\begin{align*}
& H^{i}_{c,\partial \setminus \triangle}(\rS_K(\BC), V_{\lambda}):=H^i(\rS_K^*(\BC), j_{\triangle,*} j^{\triangle}_! V_{\lambda}),\\
& H^{i}_{\triangle}(\partial \rS_K(\BC), V_{\lambda}):=H^i( \rS^*_K(\BC), i_{\triangle,*} i^{\triangle,*} j^{\triangle}_*  V_{\lambda}).
\end{align*}
We get a long exact sequence 
\begin{equation} \label{boundlongsupp}
\cdots \rightarrow H^{i}_{c,\partial \setminus \triangle}(\rS_K(\BC), V_{\lambda}) \rightarrow H^{i}(\rS_K(\BC), V_{\lambda}) \rightarrow H^{i}_{\triangle}(\partial \rS_K(\BC), V_{\lambda}) \rightarrow \cdots     
\end{equation}    
Moreover, there are canonical morphisms
\begin{equation} \label{restrbound}
i_*i^*j_* \xrightarrow{r_{\triangle}} i_{\triangle,*} i^{\triangle,*} j^{\triangle}_*
\end{equation}
such that the diagram 
\begin{center}
\begin{tabular}{c}
\xymatrix{
j_* \ar[dd] \ar[dr] & \\
& i_*i^*j_* \ar_{r_{\triangle}}[dl] \\
i_{\triangle,*} i^{\triangle,*} j^{\triangle}_*
}
\end{tabular}
\end{center}
commutes.

\begin{defi}
We define the subspaces 
\[
I^K_{\phi,\Theta} \leq I_{\phi}^K, \ I^K_{\theta \phi,\Sigma} \leq I_{\theta \phi}^K 
\]
as the ones whose elements are supported on $\Theta$, resp. on $\Sigma$.

They acquire Hodge structures 
\[
\partial H_{\Theta}^2(\phi), \  \mbox{resp.} \ \partial H_{\Sigma}^2(\theta \phi)
\]
by restriction of those defined in~\eqref{boundphiHS}, respectively~\eqref{boundphiHS_dual}.
\end{defi}

\begin{prop}\label{3stepdiag}
Let $\phi$ be a Hecke character of $F$ as in Theorem~\ref{sourcext1}. Fix a decomposition 
\[
\partial = \Theta \bigsqcup \Sigma\,.
\]
as in~\eqref{sepbound}. 
There exist extensions~\eqref{boundseq} and~\eqref{boundseq_d}  
\begin{equation} \label{boundseq_supp}
0 \rightarrow H^2(\phi)_! \rightarrow \tilde{E} \rightarrow \partial H_{\Theta}^2(\phi) \rightarrow 0\,,
\end{equation}
\begin{equation} \label{boundseq_supp_d}
0 \rightarrow \partial H_{\Sigma}^1(\theta \phi) \rightarrow \tilde{E}^\prime \rightarrow  H^2(\phi)_! \rightarrow 0
\end{equation}
and a canonical mixed $\overline{\BQ}$-Hodge structure $\tilde{E}^{\sharp}$ fitting in a commutative diagram of mixed $\overline{\BQ}$-Hodge structures, with exact rows and columns
\begin{center}
\begin{tabular}{c}
\xymatrix{
& 0 \ar[d] & 0 \ar[d] & \\
 & \partial H_{\Sigma}^1(\theta \phi) \ar@{=}[r] \ar[d] & \partial H_{\Sigma}^1(\theta \phi) \ar[d] & \\
0 \ar[r] & \tilde{E}^{\prime} \ar[d] \ar[r] & \tilde{E}^{\sharp} \ar[d] \ar[r] & \partial H_{\Theta}^2(\phi) \ar@{=}[d] \ar[r] & 0 \\
0 \ar[r] & H^2(\phi)_! \ar[r] \ar[d] & \tilde{E} \ar[r] \ar[d] & \partial H_{\Theta}^2(\phi) \ar[r] & 0 \\
& 0 & 0
}
\end{tabular}
\end{center}
\end{prop}
\begin{proof}
Consider the exact sequence 
\begin{equation} \label{boundlongsuppex}
H^{1}_{\Sigma}(\partial \rS_K(\BC), V_{k,\overline{\BQ}}) \rightarrow H^{2}_{c,\Theta}(\rS_K(\BC), V_{k,\overline{\BQ}}) \rightarrow H^{2}(\rS_K(\BC), V_k) \rightarrow H^{2}_{\Sigma}(\partial \rS_K(\BC), V_{k,\overline{\BQ}})
\end{equation}
obtained from the extension of scalars of the long exact sequence~\eqref{boundlongsupp}. We have inclusions 
\[
I^K_{\phi,\Theta} \hookrightarrow H^{2}_{\Theta}(\partial \rS_K(\BC), V_{k,\overline{\BQ}}) , \ \mbox{resp.} \ I^K_{\theta \phi,\Sigma} \hookrightarrow H^{1}_{\Sigma}(\partial \rS_K(\BC), V_{k,\overline{\BQ}})
\]
which factor as the composition of
\[
I_{\phi,\Theta}^K \hookrightarrow H^{2}(\partial \rS_K(\BC), V_{k,\overline{\BQ}}), \ \mbox{resp.} \ I_{\theta \phi, \Sigma}^K \hookrightarrow H^{1}(\partial \rS_K(\BC), V_{k,\overline{\BQ}})
\]
and of the morphisms induced by~\eqref{restrbound}. By semi-simplicity of $H^{1}_{\Sigma}(\partial S_K(\BC), V_{k,\overline{\BQ}})$), there exists a projection 
\[
H^{1}_{\Sigma}(\partial \rS_K(\BC), V_{k,\overline{\BQ}}) \twoheadrightarrow I^K_{\theta \phi,\Sigma}\,.
\]
Taking pushout along the latter, the sequence~\eqref{boundlongsuppex} yields an exact sequence 
\[
0 \rightarrow \partial H_{\Sigma}^1(\theta \phi) \rightarrow \tilde{\tilde{E}}_{\Sigma} \rightarrow H^{2}(\rS_K(\BC), V_{k,\overline{\BQ}})\,.
\]
Denoting by $\tilde{\tilde{E}}$ the pullback of 
\[
H^{2}(\rS_K(\BC), V_{k,\overline{\BQ}}) \rightarrow H^{2}(\partial \rS_K(\BC), V_{k,\overline{\BQ}})
\]
along 
\[
I^K_{\phi,\Theta} \hookrightarrow H^{2}(\partial \rS_K(\BC), V_{k,\overline{\BQ}})
\]
we obtain, by pullback via $\tilde{\tilde{E}} \rightarrow H^{2}(\rS_K(\BC), V_{k,\overline{\BQ}})$, an extension
\begin{equation}\label{penext}
0 \rightarrow \partial H_{\Sigma}^1(\theta \phi) \rightarrow \tilde{E}_{\Sigma} \rightarrow \tilde{\tilde{E}} \rightarrow 0\,.
\end{equation}
An analogous pullback applied to the extension~\eqref{boundseq} yields an extension $\tilde{E}$ as in~\eqref{boundseq_supp}. Pullback of~\eqref{penext} along the induced morphism
\[
\tilde{E} \rightarrow \tilde{\tilde{E}}
\]
provides an extension
\[
0 \rightarrow \partial H_{\Sigma}^1(\theta \phi) \rightarrow \tilde{E}^{\sharp} \rightarrow \tilde{E} \rightarrow 0\,.
\]
By construction, this extension fits in the diagram given in the statement, as its right four-term column.   
\end{proof}

\begin{coro}\label{3stepdiag_hodge}
Let $\phi$ be a Hecke character of $F$ as in Theorem~\ref{sourcext1}. Fix a decomposition
\[
\partial = \Theta \bigsqcup \Sigma
\]
as in~\eqref{sepbound} and let $m(\phi,K,\Theta)$ be the dimension of $I_{\phi,\Theta}^K$ and $m^\prime(\phi,K,\Sigma)$ the dimension of $I_{\theta \phi,\Sigma}^K$. There exists a canonical diagram of mixed $\overline{\BQ}$-Hodge structures
\begin{center}
\begin{tabular}{c}
\xymatrix{
& 0 \ar[d] & 0 \ar[d] & \\
 & \one(1)^{\oplus m^\prime(\phi,K,\Sigma)} \ar@{=}[r] \ar[d] & \one(1)^{\oplus m^\prime(\phi,K,\Sigma)} \ar[d] & \\
0 \ar[r] & E^{\prime} \ar[d] \ar[r] & E^{\sharp} \ar[d] \ar[r] & \one^{\oplus m(\phi,K,\Theta)} \ar@{=}[d] \ar[r] & 0 \\
0 \ar[r] & H_{\phi}(-1) \ar[r] \ar[d] & E \ar[r] \ar[d] & \one^{\oplus m(\phi,K,\Theta)} \ar[r] & 0 \\
& 0 & 0
}
\end{tabular}
\end{center}
\end{coro}
\begin{proof}
By Corollary~\ref{boundphiHS_sub}, the pure $\overline{\BQ}$-Hodge structure $\partial H_{\Theta}^2(\phi)$ is a direct sum of copies of the 1-dimensional $\overline{\BQ}$-Hodge structure $I_{k,-k-3,\overline{\BQ}}$. By tensoring the diagram of Proposition~\ref{3stepdiag} by the dual of $I_{k,-k-3,\overline{\BQ}}$, we obtain the statement.
\end{proof}

By a similar reasoning to the one leading to Corollaries~\ref{coroext} and~\ref{coroext_dual}, we obtain the following.
\begin{coro}
Let $\phi$ be a Hecke character of $F$ as in Theorem~\ref{sourcext1}. For any decomposition
\[
\partial = \Theta \bigsqcup \Sigma\,
\]
as in~\eqref{sepbound} and any 
\[
(\Psi, \Psi^\prime) \in I_{\phi,\Theta}^K \times I_{\theta \phi,\Sigma}^K
\]
the diagram of Corollary~\ref{3stepdiag_hodge} induces a biextension of geometric origin
\begin{equation}
B_{\phi}(\Psi, \Psi^\prime)
\end{equation}
of $H_{\phi}(-1)$.
\end{coro}

\section*{Acknowledgements}
The authors would like to thank the following institutions for their hospitality and excellent working environments: Georg-August-Universit\"at G\"ottingen, Technische Universit\"at Dresden, Mathematisches Forschungsinstitut Oberwolfach (MFO), Max Planck Institute for Mathematics in Bonn, Institut de Recherche Math\'ematique Avanc\'ee (IRMA) in Strasbourg, Laboratoire de Math\'ematiques d'Orsay (LMO), Institut de Math\'ematiques de Bourgogne (IMB) in Dijon, and Institut des Hautes \'Etudes Scientifiques (IHES) in Paris. Much of the discussion and work on this project was carried out at these institutions.

The authors also thank G. Ancona, L. Clozel, D. Frățilă, J. Fresán, V. Hernandez and S.-W. Zhang for their interest in this work and for useful comments and discussions. Very special thanks go to G\"unter Harder for sharing his vision and his ideas, for constant encouragement and for several discussions on the subject during the writing of this article. We dedicate this paper to his memory.

\section*{Conflict of interest statement} On behalf of all authors, the corresponding author states that there is no conflict of interest.

\nocite{}
\bibliographystyle{abbrv}
\bibliography{BC}
\end{document}